\documentclass[a4paper,12pt]{amsart}

\newcommand{\Cinf}{C^\infty}

\newcommand{\diff}{\mathrm{d}}
\newcommand{\pr}{\mathrm{pr}}

\newcommand{\n}[1]{\left\Vert #1\right\Vert} % NORMI

\newcommand{\R}{\mathbb{R}}

\newcommand{\N}{\mathbb{N}}

\newcommand{\be}{\begin{equation}}
\newcommand{\ee}{\end{equation}}
\newcommand{\bea}{\begin{eqnarray}}
\newcommand{\eea}{\end{eqnarray}}
\newcommand{\ben}{\begin{displaymath}}
\newcommand{\een}{\end{displaymath}}
\newcommand{\bean}{\begin{eqnarray*}}
\newcommand{\eean}{\end{eqnarray*}}
\def\ol#1{\overline{#1}}

\newcommand{\mc}[1]{\mathcal{#1}}

\newcommand{\di}[2]{\frac{\diff #1}{\diff #2}}
\newcommand{\dif}[1]{\frac{\diff}{\diff #1}}

\newcommand{\mscr}[1]{\textrm{\fontencoding{OMS}\fontfamily{ztmcm}\selectfont#1}
}

\usepackage{amsmath}
\usepackage{amssymb}
\usepackage{amsthm}
\usepackage[toc,page]{appendix}
\usepackage[english]{babel}
\usepackage{verbatim}
\usepackage[T1]{fontenc}
\usepackage{fancyhdr}
\usepackage[hmargin=2.5cm,vmargin=2.5cm]{geometry}
\usepackage{ifthen}
\usepackage{color}
\usepackage{xypic}
\input xy
\xyoption{all}

\newcommand{\doo}{\partial}

\newcommand{\Rol}{\mathsf{Rol}}
\newcommand{\wRol}{\widetilde{\mathsf{Rol}}}

\newcommand{\ctr}{\rfloor}
\newcommand{\VF}{\mathrm{VF}}

\newcommand{\so}{\mathfrak{so}}

\renewcommand{\di}[2]{\frac{\mathrm{d}{#1}}{\mathrm{d}{#2}}}

\newcommand{\hM}{\hat{M}}
\newcommand{\hg}{\hat{g}}

\newcommand{\RDist}{\mc{D}_{\mathrm{R}}}
\newcommand{\LNSD}{\mscr{L}_{\mathrm{NS}}}

\newcommand{\LRD}{\mscr{L}_{\mathrm{R}}}

\newcommand{\Sym}{\mathrm{Sym}}
\newcommand{\ISym}{\mathrm{InnSym}}

\def\[#1\]{\begin{align*}#1\end{align*}}

\newtheoremstyle{theorem}{0.5cm}{\topsep}%
   {\sffamily}%         Body font
   {}%         Indent amount (empty = no indent, \parindent = para indent)
   {\bfseries}% Thm head font
   {}%        Punctuation after thm head
   {2ex}%     Space after thm head (\newline = linebreak)
   {\thmname{#1}\thmnumber{ #2}\thmnote{ #3}}%         Thm head spec
\theoremstyle{theorem}
\newtheorem{theorem}{Theorem}[section]

\newtheorem{proposition}[theorem]{Proposition}
\newtheorem{example}[theorem]{Example}
\newtheorem{corollary}[theorem]{Corollary}
\newtheorem{remark}[theorem]{Remark}
\newtheorem{definition}[theorem]{Definition}
\newtheorem{lemma}[theorem]{Lemma}

\begin{document}

\title{Symmetries of the Rolling Model
%\thanks{
%The work of the author is supported by Finnish Academy of Sciences and Letters, KAUTE foundation
%and Institut fran\c{c}ais de Finlande.
%}
}

\author[Y. Chitour, M. Godoy M., P. Kokkonen]{Yacine Chitour\\Mauricio Godoy Molina\\Petri Kokkonen}

\thanks{The work of the first author is supported by the ANR project GCM, program ``Blanche'', (project number NT09$\_$504490) and the DIGITEO-R\'egion Ile-de-France project CONGEO. The work of the second author is partially supported by the ERC Starting Grant 2009 GeCoMethods. The work of the third author is supported by Finnish Academy of Science and Letters,
KAUTE Foundation and l'Institut fran\c{c}ais de Finlande.}



\address{L2S, Universit\'e Paris-Sud XI, CNRS and Sup\'elec, Gif-sur-Yvette, 91192, France. }
\email{yacine.chitour@lss.supelec.fr}

\address{L2S, Universit\'e Paris-Sud XI, CNRS and Sup\'elec, Gif-sur-Yvette, 91192, France. }
\email{mauricio.godoy@gmail.com}

\address{L2S, Universit\'e Paris-Sud XI, CNRS and Sup\'elec, Gif-sur-Yvette, 91192, France and University of Eastern Finland, Department of Applied Physics, 70211, Kuopio, Finland.}
\email{petri.kokkonen@lss.supelec.fr}

\maketitle
 
\begin{abstract}
In the present paper, we study the infinitesimal symmetries of the model of two Riemannian manifolds $(M,g)$ and $(\hM,\hat g)$ rolling without twisting or slipping. We show that, under certain genericity hypotheses, the natural bundle projection from the state space $Q$ of the rolling model onto $M$ is a principal bundle if and only if $\hM$ has constant sectional curvature. Additionally, we prove that when $M$ and $\hM$ have different constant sectional curvatures and dimension $n\geq3$, the rolling distribution is never flat, contrary to the two dimensional situation of rolling two spheres of radii in the proportion $1\colon3$, which is a well-known system satisfying \'E. Cartan's flatness condition.
\end{abstract}
  
\tableofcontents

%\newpage

%%%%%%%%%%%%%%%%%%%%%%%%%%%%%%%%%%%%%
\section{Introduction}
%%%%%%%%%%%%%%%%%%%%%%%%%%%%%%%%%%%%%

A very old and difficult problem in differential geometry is the study of symmetries of distributions. A seminal contribution is the celebrated paper by \'E. Cartan~\cite{cartan10} in which, in modern terms, he studied distributions of rank two on a manifold of dimension five and, more precisely, the associated equivalence problem. Recall that a rank $l$ vector distribution $D$ on an $n$-dimensional manifold $M$ or $(l, n)$-distribution (where $l < n$) is, by definition, an $l$-dimensional subbundle of the tangent bundle $TM$, i.e., a smooth assignment $q\mapsto D|_q$ defined on $M$ where $D|_q$ is an $l$-dimensional subspace of the tangent space $T_qM$. Two vector distributions $D_1$ and $D_2$ are said to be equivalent, if there exists a diffeomorphism $F : M \rightarrow M$ such that $F_*D_1|_q = D_2|_{F(q)}$ for every $q\in M$. Local equivalence of two distributions is defined analogously.
%Two germs of vector distributions $D_1$ and $D_2$ at $q_0\in M$ are said to be equivalent, if there exist
%neighborhoods $U$ and $U'$ of $q_0$ and a diffeomorphism $F : U\rightarrow U'$such that
%$F_*D_1(q) = D_2(F(q))$, for every $q\in U$ and $F(q_0) = q_0$.
The equivalence problem consists in constructing invariants of distributions with respect to the equivalence relation defined above. 
The main contribution of Cartan in \cite{cartan10}  was the introduction of the ``reduction-prolongation'' procedure for building invariants and the characterization for $(2,5)$-distributions via a functional invariant (Cartan's tensor) which vanishes precisely when the distribution is flat, that is, when it is locally equivalent to the (unique) graded nilpotent Lie algebra of step $3$ with growth vector $(2,3,5)$. In this case, the Lie algebra of symmetries of the distribution corresponds to the $14$-dimensional Lie algebra ${\mathfrak{g}}_2$ and this situation is maximal, that is, in the non-flat case the dimension of the Lie algebra of symmetries is strictly less than $14$.  In fact, Cartan  gave a geometric description of the flat $G_2$-structure as the differential system that describes space curves of constant torsion $2$ or $1/2$ in the standard unit $3$-sphere (see Section 53 in Paragraph XI in \cite{cartan10}.) It has been a folkloric fact among the control theory community that the flat situation described above occurs in the problem of two $2$-dimensional spheres rolling one against the other without slipping or spinning, assuming that the ratio of their radii is $1\colon3$, see~\cite{bor06} for some historical notes and a thorough attempt of an explanation for this ratio. In fact, whenever the ratio of their radii is different from $1\colon 3$, the Lie algebra of symmetries becomes ${\mathfrak{so}}(3)\times{\mathfrak{so}}(3)$, thus dropping its dimension to 6. A complete answer to this strange phenomenon as well as a geometric reason for Cartan's tensor was finally given in two remarkable papers~\cite{zelenko061,zelenko062} (cf. also \cite{agrachev07}), where a geometric method for construction of functional invariants of generic germs of $(2, n)$-distribution for arbitrary $n\geq5$ is developed. It has been recently observed in~\cite{nurowski12} that the Lie algebra of symmetries of a system of rolling surfaces can be $\mathfrak{g}_2$ in the case of non-constant Gaussian curvature.

As for the rolling model, its two dimensional version has been intensively studied by the control community for quite a while, see for example~\cite{ACL,agrachev99,bryant-hsu,chelouah01,jurd1,marigo-bicchi,murray-sastry}. Indeed, the mechanical problem of a sphere rolling can be traced back to the 19th century, in two seminal papers by S. A. Chaplygin~\cite{Chap1,Chap2}, recently translated. It was not until the publication of the book~\cite{sharpe97} that the higher dimensional problem became better known to the control theorists, though it had been introduced several years before in~\cite{nomizu78}. A major disadvantage of Sharpe's definition was the use of submanifolds of Euclidean space, with a strong dependence on their concrete realizations, nevertheless it still yield some interesting results, for example~\cite{jurdjevic08}. Trying to deal with this inconvenience was the starting point of the studies~\cite{arxiv,norway} in which a coordinate-free model for the rolling dynamics was introduced, where the restrictions of no-twist and no-slip were encoded in terms of the so-called rolling distribution $\RDist$. Recently non-trivial extensions to manifolds with different dimensions~\cite{CK2}, semi-Riemannian manifolds~\cite{ML}, and Cartan geometries~\cite{CGK2} have been presented. Besides geometric issues that are associated to the intrinsic definition for the rolling model (e.g., the question of existence of such dynamics~\cite{GG}), one can address the problem of finding conditions on the pair of manifolds $M$ and $\hat{M}$ so that the rolling model is completely controllable, i.e., if $Q$ denotes the state space of the model of two Riemannian manifolds $M$ and $\hM$ rolling without slipping or spinning, one says that the associated rolling model is completely controllable if, given arbitrary $q_0,q_1\in Q$, one can roll $\hat{M}$ on $M$ without slipping or spinning from the initial position $q_0$ to the final position $q_1$. That typical issue of control theory is usually solved by evaluating, at every point $q\in Q$, the Lie algebra generated by the distribution $\RDist$. It turns out that this approach is almost impossible to carry over for the general $n$-dimensional rolling model (cf. \cite{CK2}) except for $n=3$. On the other hand, when one of the manifolds has constant sectional curvature, the distribution $\RDist$ is a principal bundle connection for the canonical projection map $\pi_{Q,M}\colon Q\to M$ and that key feature enables one to successfully address 
the controllability issue ``without Lie brackets computations'' because the latter reduces to the determination of a certain holonomy group associated to an appropriate linear connection~\cite{CK,CGK}.

In this paper, we study the Lie algebra of symmetries of the rolling distribution $\RDist$ over the state space $Q$. We obtain as consequences of this analysis the answers of two problems arising from the issues above mentioned. The first of these says that, under certain genericity assumptions on $M$ and $\hat{M}$, the distribution $\RDist$ is a principal bundle connection for $\pi_{Q,M}$ if and only if $\hM$ has constant sectional curvature. Our second main result refers to the question of flatness of the rolling distribution for the case of spaces of constant curvature. In this context, a regular distribution of rank $k$ on a manifold of dimension $n$ is said to be flat if it is locally equivalent to its nilpotent approximation. We prove that, as long as the curvatures of $M$ and $\hM$ are different, the rolling distribution is never flat in dimensions $\geq3$, contrary to what happens for the $1:3$ phenomenon in two dimensions described previously.

The paper is organized as follows. In Section~\ref{sec:not} we introduce the basic terminology concerning the higher dimensional rolling problem that will be used throughout the paper. Section~\ref{sec:symm} starts the study of the symmetries of the rolling model, addressing later the restricted case of inner symmetries, that is, symmetries induced by vector fields in the rolling distribution. Section~\ref{sec:principal} presents the first of our main results mentioned above. A key tool in this section is the set $\Sym_0(\RDist)$ of symmetries that lie in the kernel of the differential $(\pi_{Q,M})_*$. In fact, the aforementioned result follows from a complete characterization of $\Sym_0(\RDist)$ as the symmetries induced by the Killing vector fields of $\hM$. In Section~\ref{sec:flat} we present the second main result mentioned in the previous paragraph. We begin by studying the nilpotent approximation of the rolling distribution, from which we can deduce its non-flatness if the dimension is greater than 3.

%%%%%%%%%%%%%%%%%%%%%%%%%%%%%%%%%%%%%
\section{Notations and terminology}\label{sec:not}
%%%%%%%%%%%%%%%%%%%%%%%%%%%%%%%%%%%%%

If $\mc{D}$ is a smooth constant rank distribution on $M$, we write $\VF_{\mc{D}}$ for the set of $X\in\VF(M)$ such that $X|_x\in\mc{D}|_x$ for all $x\in M$.
If $N$ is a submanifold of $M$, then we say that $\mc{D}$ is \emph{tangent to $N$},
if $\mc{D}|_x\subset T|_x N$ for all $x\in N$.

\begin{definition}\label{def:sym}
Let $\mc{D}$ be a smooth distribution of constant rank on $M$. Then $X\in\VF(M)$ is called an \emph{infinitesimal symmetry} of $\mc{D}$ if $[X,\VF_{\mc{D}}]\subset\VF_{\mc{D}}$. The vector space of all the infinitesimal symmetries of $\mc{D}$ is denoted by $\Sym(\mc{D})$.

An infinitesimal symmetry $X\in\Sym(\mc{D})$ is called an \emph{inner infinitesimal symmetry} if $X\in\VF_{\mc{D}}$. The set of all inner infinitesimal symmetries is denoted by $\ISym(\mc{D})$.
\end{definition}

\begin{remark}
The set $\ISym(\mc{D})$ is a vector subspace of $\VF_{\mc{D}}$, given by $\ISym(\mc{D})=\Sym(\mc{D})\cap\VF_{\mc{D}}$.
\end{remark}

\begin{definition}
For a distribution $\mc{D}$ on $M$, we define the $\mc{D}$-orbit of $x\in M$, denoted by $\mc{O}_{\mc{D}}(x)$ as the set of all points in $M$ that can be connected to $x$ by an absolutely continuous curve with velocity almost everywhere contained in $\mc{D}$. 
\end{definition}

\begin{remark}
By Nagano-Sussman's theorem, see~\cite{agrachev04}, the orbit $\mc{O}_{\mc{D}}(x)$ is an immersed submanifold of $M$.
\end{remark}

As an abbreviation, we usually refer to infinitesimal symmetries (resp. infinitesimal inner symmetries) of $\mc{D}$ simply as symmetries (resp. inner symmetries) of $\mc{D}$.

For the sake of completeness, we recall some of the terminology for the model of two Riemannian manifolds, one rolling against the other without twisting or slipping, introduced in~\cite{arxiv,CK}. For a more detailed discussion, we refer the reader to~\cite{arxiv}. Let $(M,g)$ and $(\hM,\hg)$ be two oriented connected Riemannian manifolds. The state space $Q=Q(M,\hM)$ of the rolling model is the manifold
\[
Q=Q(M,\hat{M})=\big\{A:T|_x M\to T|_{\hat{x}} \hat{M}\ |&\ x\in M,\ \hat{x}\in\hat{M}, \\
&\ A\ \textrm{linear isometry}, \ \det(A)>0\big\}.
\]
%where ``o-isometry'' means ``orientation preserving isometry''.

Given a point $q=(x,\hat{x};A)\in Q$, a vector $\ol{X}=(X,\hat{X})\in T|_{(x,\hat{x})}(M\times\hM)$ and any smooth curve $t\mapsto \ol{\gamma}(t)=(\gamma(t),\hat{\gamma}(t))$ in $M\times \hat{M}$ defined on an open interval $I\ni 0$ such that $\ol{\gamma}(0)=(x,\hat{x})$, and $\dot{\ol{\gamma}}(0)=\ol{X}$, the \emph{no-spinning lift} of $\ol{X}$ at $q$
if defined by
\begin{align}\label{eq:2.1:1}
\LNSD(\ol{X})|_q:=\dif{t}\big|_0 \big(P^t_0(\hat{\gamma})\circ A\circ P^0_t(\gamma)\big)\in T|_{q}Q,
\end{align}
where $P^b_a(\gamma)$ (resp. $P_a^b(\hat{\gamma})$)
denotes the parallel transport map along $\gamma$ from $\gamma(a)$ to $\gamma(b)$
(resp. along $\hat{\gamma}$ from $\hat{\gamma}(a)$ to $\hat{\gamma}(b)$).
It is readily seen that the definition of $\LNSD(\ol{X})|_q$ does not depend on the choice of the smooth
curve $\ol{\gamma}$ as long as it satisfies $\ol{\gamma}(0)=(x,\hat{x})$ and $\dot{\ol{\gamma}}(0)=\ol{X}$.

Similarly, we define the \emph{rolling lift} of $X\in T|_x M$ to $q=(x,\hat{x};A)\in Q$ as
\begin{align}\label{eq:2.5:3}
\LRD(X)|_q:=\LNSD(X, AX)|_q.
\end{align}
Notice that $\LRD$ also defines a natural map $\LRD:\VF(M)\to \VF(Q)$ such that $\LRD(X):=\big(q\mapsto \LRD(X)|_q\big)$.

\begin{definition}\label{def:2.5:1}\label{def:rdist}
The \emph{rolling distribution} $\RDist$ on $Q$ is the $n$-dimensional smooth distribution defined, for $(x,\hat{x};A)\in Q$, by
\begin{align}\label{eq:2.5:1}
\RDist|_{(x,\hat{x};A)}=\LRD(T|_x M)|_{(x,\hat{x};A)}.
\end{align}
\end{definition}

An absolutely continuous curve $t\mapsto q(t)=(\gamma(t),\hat{\gamma}(t);A(t))$ in $Q$ that is almost everywhere tangent to $\RDist$ is called a \emph{rolling curve}. This condition can be rewritten as $\dot{q}(t)=\LRD(\dot{\gamma}(t))|_{q(t)}$, a.e. $t$. It was shown in~\cite{arxiv,norway} that such curves are exactly those that describe the dynamics of rolling $M$ against $\hM$ without twisting or spinning. 

As it can be noticed already, there are several fiber and vector bundles that will play an important role in the main results of this article. As an abuse of notation, we will often denote the bundles only by its projection maps. The fiber bundles $\pi_Q:Q\to M\times\hM$ and $\pi_{Q,M}:Q\to M$ are the projections $\pi_Q(x,\hat x;A)=(x,\hat x)$ and $\pi_{Q,M}(x,\hat x;A)=x$. Observe that $\pi_Q$ is a fiber subbundle of the vector bundle $\pi_{T^*M\otimes T\hM}\colon T^*M\otimes T\hM\to M\times\hM$. For any manifold $N$, the map $\pi_{T_m^kN}\colon T_m^kN\to N$ denotes the vector bundle of $(k,m)$-tensors on $N$, and the special case $(k,m)=(1,0)$ for the tangent bundle is simply denoted by $\pi_{TN}$.
Given two fiber bundles $\xi$ and $\eta$ over the same manifold $M$, we denote by $\Cinf(\xi,\eta)$ the space of smooth bundle maps from $\xi$ to $\eta$.
Assuming that $\xi$ and $\eta$ are vector bundles,
for $x\in M$ and $f\in\Cinf(\xi,\eta)$, one defines the vertical derivative $\nu(w)|_u(f)$ of $f$ at $u\in\xi^{-1}(x)$ in the direction of $w\in\xi^{-1}(x)$ as
\[
\nu(w)|_u(f)=\left.\di{}{t}\right|_{t=0}f(u+tw),
\]
which can be identified with an element of the fiber $\eta^{-1}(x)$.
This notion then immediately extends to the situation
where there is a (possibly non vector) fiber subbundle $\lambda$ of $\xi$,
and $f\in\Cinf(\lambda,\eta)$ if, moreover, $\nu(w)|_u$ is tangent to the total space of $\lambda$.

We still need to extend the notion of the vector $\LRD(X)|_q$, $q\in Q$, to an operator acting on tensor valued maps.
Suppose $N$ is a submanifold of $Q$ such that $\RDist$ is tangent to $N$,
i.e., $\RDist|_q\subset T|_q N$ for all $q\in N$.
Suppose that $F:N\to T^k_n(M\times\hat{M})$ is smooth and $\pr_1\circ \pi_{T^k_m(M\times\hat{M})}\circ F=\pi_{Q,M}$
where $\pr_1:M\times\hat{M}\to M$; $(x,\hat{x})\mapsto x$.
For $q=(x,\hat{x};A)\in N$ and $X\in T|_x M$ one defines $\LRD(X)|_q F$ as the element of
$T^k_m(M\times\hat{M})$ given by
\[
\LRD(X)|_q F:=\ol{\nabla}_{(X,AX)} F(q(t)),
\]
where $q(t)$ is any smooth curve in $Q$ such that $\dot{q}(0)=\LRD(X)|_q$ (as vectors)
and $\ol{\nabla}$ is the connection induced by the Levi-Civita connections $\nabla$ and $\hat{\nabla}$ on the bundle $\pi_{T^k_m(M\times\hat{M})}$.
Note that above $F(q(t))$ is a tensor field
along the curve $(\gamma(t),\hat{\gamma}(t)):=\pi_Q(q(t))$,
whose initial velocity is $(\dot{\gamma}(0),\dot{\hat{\gamma}}(0))=(X,AX)$,
so that the expression $\ol{\nabla}_{(X,AX)} F(q(t))$ makes sense.
Moreover, this expression is independent of the choice of the smooth curve $q(t)$ as long as $\dot{q}(0)=\LRD(X)|_q$
(e.g. $q(t)$ could be taken as a rolling curve).

%to be the unique element of $T^k_m(M\times\hat{M})$
%such that for any section $\ol{\omega}$ of $\pi_{T^m_k(M\times\hat{M})}$ (the dual bundle of $\pi_{T^k_m(M\times\hat{M})}$)
%\[
%(\LRD(X)|_q F)\ctr\ol{\omega}=\LRD(X)|_q (F\ctr\ol{\omega})-\ol{\nabla}_{(X,AX)} \ol{\omega}
%\]
%where we denoted, for clarity, by $\ctr$ the contraction of tensors.

For an inner product space $(V,\langle\cdot\,,\cdot\rangle)$, denote by $\so(V)$ the Lie algebra of skew-symmetric endomorphisms of $V$ with respect to $\langle\cdot\,,\cdot\rangle$. For the rest of the paper, given $x\in M$, we identify the vector space $\bigwedge^2 T|_x M$ with $\so(T|_x M)$ as follows: If $X,Y,Z\in T|_x M$, then
\[
(X\wedge Y)Z=g(Z,Y)X-g(Z,X)Y.
\]

To conclude this section, we present a convenient result that allows us to compute Lie brackets of vector fields on $Q$. Its proof follows after a careful calculation, and the details can be found in~\cite{arxiv}. 

\begin{proposition}
Let $\nabla$ and $\hat\nabla$ be the Levi-Civita connections of $M$ and $\hM$ respectively. Let $\ol{T}=(T,\hat{T}),\ol{S}=(S,\hat{S})\in\Cinf(\pi_Q,\pi_{T(M\times\hat{M})})$ and $\ol{U},\ol{V}\in\Cinf(\pi_Q,\pi_{T^*M\otimes T\hat{M}})$ be such that $\ol{U}(q),\ol{V}(q)\in A\so(T|_x M)$ for all $q=(x,\hat{x};A)\in Q$.
Then if
\[
\mc{X}|_q:=\LNSD(\ol{T}(q))|_q+\nu(\ol{U}(q))|_q,
\quad
\mc{Y}|_q:=\LNSD(\ol{S}(q))|_q+\nu(\ol{V}(q))|_q,
\]
one has
%, as operators acting on $\Cinf(\pi_Q,\pi_{T^k_m(M\times\hat{M})})$,
\[
[\mc{X},\mc{Y}]|_q=&\LNSD(\mc{X}|_q \ol{S}-\mc{Y}|_q\ol{T})|_q+\nu(\mc{X}|_q \ol{V}-\mc{Y}|_q\ol{U})|_q \\
&+\nu(AR(T(q),S(q))-\hat{R}(\hat{T}(q),\hat{S}(q))A)|_q+R^{\ol{\nabla}}(\ol{T}(q),\ol{S}(q)),
\]
where $R$ and $\hat R$ are the Riemannian curvatures of $(M,g)$ and $(\hM,\hat{g})$ respectively, and $R^{\ol{\nabla}}$ is the curvature of the connection $\ol\nabla$.
\end{proposition}

\begin{remark}
The above proposition holds true if one replaces everywhere $Q$ by any submanifold $N\subset Q$
such that $\RDist$ is tangent to it, replacing the condition that $\ol{U}(q),\ol{V}(q)\in A\so(T|_x M)$ for $q\in Q$
by the assumption that $\mc{X},\mc{Y}$ be tangent to $N$.
\end{remark}

%%%%%%%%%%%%%%%%%%%%%%%%%%%%%%%%%%%%%
\section{Symmetries of the Rolling Distribution}\label{sec:symm}
%%%%%%%%%%%%%%%%%%%%%%%%%%%%%%%%%%%%%

\subsection{General Symmetries}

We begin our study of the symmetries of the rolling model by finding a condition, equivalent to the one in Definition~\ref{def:sym}, for a vector field $S\in\VF(Q)$ to be a symmetry.

\begin{proposition}\label{pr:symeq}
Let $Z\in \Cinf(\pi_{Q,M},\pi_{TM})$, $\hat{Z}\in\Cinf(\pi_{Q,\hat{M}},\pi_{T\hat{M}})$,
$\ol{U}\in\Cinf(\pi_Q,\pi_{T^*M\otimes T\hat{M}})$
be such that $\ol{U}(q)\in A\so(T|_x M)$ for all $q=(x,\hat{x};A)\in Q$.
Defining
\[
S|_q:=\LNSD(Z(q),\hat{Z}(q))|_q+\nu(\ol{U}(q))|_q,
\]
then $S\in \Sym(\RDist)$ if and only if for all $q=(x,\hat{x};A)\in Q$ and all $X\in T|_x M$,
one has
\begin{align}
\ol{U}(q)X&=-A\LRD(X)|_q Z+\LRD(X)|_q\hat{Z} \label{eq:sym:1} \\
\LRD(X)|_q\ol{U}&=-AR(X\wedge Z(q))+\hat{R}(AX\wedge \hat{Z}(q))A \label{eq:sym:2}
\end{align}
\end{proposition}

\begin{proof}
If $X\in\VF(M)$, then
\begin{align}
[S,\LRD(X)]|_q=&\LRD(\nabla_{Z(q)} X)|_q-\LNSD(\LRD(X)|_q Z,\LRD(X)|_q\hat{Z})|_q \nonumber\\
&+\nu(AR(Z(q),X)-\hat{R}(\hat{Z}(q),AX)A)|_q \nonumber\\
&+\LNSD(0,\ol{U}(q)X)|_q-\nu(\LRD(X)|_q\ol{U})|_q \nonumber\\
=&\LRD(\nabla_{Z(q)} X-\LRD(X)|_qZ)|_q \nonumber\\
&+\LNSD(0,\ol{U}(q)X+A\LRD(X)|_q Z-\LRD(X)|_q\hat{Z}(q))|_q \label{eq:LNvanish}\\
&+\nu(-\LRD(X)|_q\ol{U}+AR(Z(q),X)-\hat{R}(\hat{Z}(q),AX)A)|_q.\label{eq:nuvanish}
\end{align}
Note that $S\in\Sym(\RDist)$ if and only if $[S,\LRD(X)]\in\RDist$, for all $X\in\VF(M)$. Hence, $S\in\Sym(\RDist)$ if and only if the terms~\eqref{eq:LNvanish},\eqref{eq:nuvanish} above vanish for every $X\in\VF(M)$.
\end{proof}

We will often use the notation
\[
S_{(Z,\hat{Z},\ol{U})}|_q:=\LNSD(Z(q),\hat{Z}(q))|_q+\nu(\ol{U}(q))|_q.
\]
In~\cite{CK} a notion of curvature especially adapted to the rolling model was introduced. This idea will play a fundamental role in the subsequent developments, so we briefly recall it here for the sake of completeness.
%Given vector fields $X,Y\in\VF(M)$, the \emph{rolling curvature endomorphism} $\Rol(X,Y)\colon\pi_{T^*M\otimes T\hM}\to\pi_{T^*M\otimes T\hM}$ is a map defined by
%\[
%\Rol(X,Y)(A):=AR(X,Y)-\hat R(X,Y)A.
%\]
For $q=(x,\hat x;A)\in Q$, the \emph{rolling curvature} is the linear map
\[
\Rol_q\colon\bigwedge^2T|_xM\to T^*|_xM\otimes T|_{\hat x}\hM;
\quad \Rol_q(X\wedge Y):=AR(X,Y)-\hat R(X,Y)A.
\]
For convenience, we also define
\[
\wRol_q:\bigwedge^2T|_xM\to \bigwedge^2T|_xM; \quad \wRol_q(X\wedge Y)=R(X,Y)-A^{-1}\hat R(X,Y)A,
\]
i.e. $\Rol_q=A\wRol_q$.
The fact that the values of $\wRol_q$ are in $\bigwedge^2T|_xM$ instead of just $T^*M\otimes TM$,
follows from well-known properties of the curvature tensors $R,\hat{R}$.
%uniquely defined by setting $\Rol_q(X\wedge Y)=\Rol(X,Y)(A)$.

As usual, for a smooth distribution $\mc{D}$, we denote by $\mc{D}^{(k)}$ the $k$th element in the canonical flag of $\mc{D}$, that is the $k$th step in the iterative definition
\[
\mc{D}^{(1)}=\mc{D},\quad\mc{D}^{(k+1)}=\mc{D}^{(k)}+[\mc{D}^{(k)},\mc{D}].
\]

\begin{proposition}\label{pr:vertZ}
Let $S_{(Z,\hat{Z},\ol{U})}\in\Sym(\RDist)$.
Then for all $q=(x,\hat{x};A)\in Q$, $X,Y\in T|_x M$,
\[
A\nu(\Rol_q(X\wedge Y))|_q Z=\nu(\Rol_q(X\wedge Y))|_q \hat{Z}.
\]
\end{proposition}

\begin{proof}
Notice that if $S_{(Z,\hat{Z},\ol{U})}\in\Sym(\RDist)$, then $S_{(Z,\hat{Z},\ol{U})}\in\Sym(\RDist^{(k)})$
for all $k\in\N$.
Note that $\nu(\Rol(X\wedge Y))=\big(q\mapsto \nu(\Rol_q(X\wedge Y))|_q\big)$ belongs to $\RDist^{(2)}$
and
\[
[S_{(Z,\hat{Z},\ol{U})}, \nu(\Rol(X\wedge Y))]
=&-\LNSD\big(\nu(\Rol_q(X\wedge Y))|_q Z,\nu(\Rol_q(X\wedge Y))|_q \hat{Z}\big)\big|_q \\
&+\nu\big(\LNSD(Z(q),\hat{Z}(q))|_q\Rol(X\wedge Y)+\nu(\ol{U}(q))|_q\Rol(X\wedge Y)|_q\big)\big|_q \\
&-\nu\big(\nu(\Rol_q(X\wedge Y))|_q \ol{U}\big) \\
=&-\LNSD\big(0,-A\nu(\Rol_q(X\wedge Y))|_q Z+\nu(\Rol_q(X\wedge Y))|_q \hat{Z}\big)\big|_q \\
&-\LRD(\nu(\Rol_q(X\wedge Y))|_q Z)|_q+\nu(\cdots)|_q.
\]
Since $\RDist^{(2)}$ is spanned by vectors of the form $\LRD(X)$ and $\nu(\Rol(X\wedge Y))$
and since $S_{(Z,\hat{Z},\ol{U})}\in\Sym(\RDist^{(2)})$,
it follows that the $\LNSD$ term above vanishes, i.e.
\begin{equation*}
-A\nu(\Rol_q(X\wedge Y))|_q Z+\nu(\Rol_q(X\wedge Y))|_q \hat{Z}=0.\qedhere
\end{equation*}

\end{proof}

\begin{remark}
The above propositions holds true if one replaces everywhere $Q$ by a submanifold $N\subset Q$
to which $\RDist$ is tangent, one replaces the set $\Sym(\RDist)$ by $\Sym(\RDist|_{N})$,
and if the condition that $\ol{U}(q)\in A\so(T|_x M)$ for $q\in Q$
is replaced by the assumption that $S$ be tangent to $N$.
The notation $\pi_{N,M}$ (resp. $\pi_{N,\hat{M}}$, $\pi_{N}$) would mean in this context $\pi_{Q,M}|_N$
(resp.  $\pi_{Q,\hat{M}}|_N$, $\pi_{Q}|_N$).
\end{remark}

%\begin{remark}
%In the previous proposition, one can replace everywhere $Q$ by an open subset $O$ of it,
%with the understanding that $\Sym(\RDist)$ gets replaced by $\Sym(\RDist|_O)$.
%\end{remark}

%%%%%%%%%%%%%%%%%%%%%%%%%%%%%%%%%%%%
\subsection{Inner Symmetries}
%%%%%%%%%%%%%%%%%%%%%%%%%%%%%%%%%%%%

The aim of this subsection is to briefly study some basic properties
or inner symmetries of $\RDist$ as well as of $\RDist|_{\mc{O}_{\RDist}(q_0)}$ for $q_0\in Q$.
In particular, we will unveil a connection between the existence of inner symmetries of the type $\LRD(Z)$, $Z\in\VF(M)$, and one of the manifolds having constant sectional curvature.
%This fact will be fundamental in the proof of one of our main results.
%Before presenting this relation, we need a preliminary technical result.
We will begin by characterizing the inner symmetries.

\begin{proposition}\label{pr:inneq}
The following properties hold.

\begin{itemize}
\item[(i)] If $S_{(Z,\hat{Z},\ol{U})}\in\ISym(\RDist)$, then
for all $q=(x,\hat{x};A)\in Q$,
\[
& \hat{Z}(q)=AZ(q),\ \ol{U}(q)=0, \\
& \Rol_q(X\wedge Z(q))=0,\quad \forall X\in T|_x M.
\]
\item[(ii)] If there exists $Z\in\Cinf(\pi_{Q,M},\pi_{TM})$ such that
$\Rol_q(X\wedge Z(q))=0$ for all $q=(x,\hat{x};A)\in Q$ and $X\in T|_x M$,
then defining $\hat{Z}(q):=AZ(q)$,
we have that 
\[
S_{(Z,\hat{Z},0)}\in\ISym(\RDist).
\]

%\item[(iii)] If $S_{(Z,\hat{Z},\ol{U})}\in\Sym(\RDist)$
%and $\hat{Z}(q)=AZ(q)$ for all $q\in Q$, then $S_{(Z,\hat{Z},\ol{U})}\in\ISym(\RDist)$.
\end{itemize}
\end{proposition}

\begin{proof}
(i) Since
\[
S_{(Z,\hat{Z},\ol{U})}=\LRD(Z(q))|_q+\LNSD(0,-AZ(q)+\hat{Z}(q))|_q+\nu(\ol{U}(q))|_q,
\]
we see that if $S_{(Z,\hat{Z},\ol{U})}\in\ISym(\RDist)$, then 
$-AZ(q)+\hat{Z}(q)=0$ and $\ol{U}(q)=0$ for all $q\in Q$.
Then by \eqref{eq:sym:2},
\[
0=\LRD(X)|_q\ol{U}=-AR(X\wedge Z(q))+\hat{R}(X\wedge AZ(q))A=-\Rol_q(X\wedge Z(q)).
\]

(ii) Setting $\ol{U}(q)=0$,
we see that
\[
-A\LRD(X)|_q Z+\LRD(X)|_q \hat{Z}=-\LRD(X)|_q \big((\cdot)Z(\cdot)-\hat{Z}(\cdot)\big)=0=\ol{U}(q)X
\]
and
\[
 -AR(X\wedge Z(q))+\hat{R}(AX\wedge \hat{Z}(q))A
&=-AR(X\wedge Z(q))+\hat{R}(AX\wedge AZ(q))A \\
&=-\Rol_q(X\wedge Z(q))=0=\LRD(X)|_q \ol{U}.
\]
Hence the vector field $S_{(Z,\hat{Z},0)}=\LRD(Z(\cdot))$ satisfies equations \eqref{eq:sym:1}-\eqref{eq:sym:2}, in other words, we have $S_{(Z,\hat{Z},0)}\in\ISym(\RDist)$.\qedhere

%(iii) Write $\delta \ol{Z}(q):=-AZ(q)+\hat{Z}(q)$ and notice that 
%\eqref{eq:sym:1} can be written in the form
%\[
%\ol{U}(q)X=\LRD(X)|_q (\delta\ol{Z}).
%\]
%Therefore, since $\delta \ol{Z}(q)=-AZ(q)+\hat{Z}(q)=0$ for all $q\in Q$, by assumption,
%and because $S_{(Z,\hat{Z},\ol{U})}\in\Sym(\RDist)$,
%one has $\ol{U}(q)X=0$ for all $X\in T|_x M$ and hence $\ol{U}(q)=0$.
%This implies that $S_{(Z,\hat{Z},\ol{U})}=\LRD(Z(q))|_q$ and thus $S_{(Z,\hat{Z},\ol{U})}\in\ISym(\RDist)$.
\end{proof}

\begin{remark}\label{rem:infdim}
Notice that if $S\in\ISym(\RDist)$, then $fS\in\ISym(\RDist)$ for all $f\in\Cinf(Q)$.
Therefore, if $\ISym(\RDist)$ is a non-trivial space,
it has to be infinite dimensional as a vector space over $\R$.
\end{remark}

\begin{remark}
As before, one can replace in the above proposition the space $Q$ by any of its submanifolds $N$
such that $\RDist$ is tangent to $N$,
if one also replaces $\ISym(\RDist)$ by $\ISym(\RDist|_{N})$.
\end{remark}

\begin{example}
Suppose that $(M,g),(\hat{M},\hat{g})$
both have constant, equal curvature. 
Then any vector field $Y\in\VF(M)$ gives rise to an inner symmetry.
Indeed, defining $Z(q):=Y|_x$, for $q=(x,\hat{x};A)$,
one has $\Rol_q(X\wedge Z(q))=0$ for all $X\in T|_x M$.
In other words, in this setting $\LRD(\VF(M))\subset \ISym(\RDist)$.
\end{example}

Next we present the result announced at the beginning of this subsection.
%that is a partial converse of the previous example (see also the example below following the proposition).
As a notational remark, if $X,Y\in T|_xM$ is an orthonormal pair of vector,
then $\sigma_{(X,Y)}$ denotes the sectional curvature of $M$ at $x$ with respect of the plane spanned by $X$ and $Y$.
For convenience, we use $\hat\sigma$ for the analogous concept on $\hM$.

\begin{proposition}
Suppose that there exists $Z\in\VF(M)$, $Z\neq 0$, such that
$\LRD(Z)\in\ISym(\RDist)$.
Then $(\hat{M},\hat{g})$ has constant curvature $\hat{c}\in\R$
and for every $x\in M$ such that $Z|_x\neq 0$,
the sectional curvatures of $(M,g)$
along all the planes containing $Z|_x$ are equal to $\hat{c}$.
\end{proposition}

\begin{proof}
Fix $x_1\in M$ such that $Z|_{{x}_1}\neq 0$.
Then for all $q\in(\pi_{Q,M})^{-1}(x_1)$, say $q=(x_1,\hat{x};A)$, we have
\[
0=\Rol_{q}(X\wedge Z|_{x_1})=AR(X\wedge Z|_{x_1})-\hat{R}(AX\wedge AZ|_{x_1})A,
\quad \forall X\in T|_{x_1} M,
\]
from which we get, whenever $X\in T|_{x_1} M$ is a unit vector orthogonal to $Z|_{x_1}$,
\[
\sigma_{\left(X,Z|_{x_1}/\n{Z|_{x_1}}_g\right)}=\hat{\sigma}_{\left(AX,AZ|_{x_1}/\n{Z|_{x_1}}_g\right)}.
\]

Let us thus fix a unit vector $X\in T|_{x_1} M$ orthogonal to $Z|_{x_1}$.
Given any $\hat{x}\in\hat{M}$ and orthonormal pair of vectors $\hat{X},\hat{Y}\in T|_{\hat{x}}\hat{M}$,
there exists a $A\in Q|_{(x_1,\hat{x})}$
such that $AX=\hat{X}$, $AZ|_{x_1}/\n{Z|_{x_1}}_g=\hat{Y}$,
and therefore,
\[
\hat{\sigma}_{\left(\hat{X},\hat{Y}\right)}=\sigma_{\left(X,Z|_{x_1}/\n{Z|_{x_1}}_g\right)}.
\]
Since $\hat{x}\in \hat{M}$ and $\hat{X},\hat{Y}\in T|_{\hat{x}}\hat{M}$
were an arbitrary point and an arbitrary orthonormal pair of vectors, it follows that the sectional curvatures of $(\hat{M},\hat{g})$
are all equal to $\hat{c}:=\sigma_{(X,Z|_{x_1}/\n{Z|_{x_1}}_g)}$.

Now if $x\in M$ is such that $Z|_x\neq 0$,
then again, for any unit vector $X\in T|_x M$ orthogonal to $Z|_x$,
\[
\sigma_{(X,Z|_{x}/\n{Z|_{x}}_g)}=\hat{\sigma}_{(AX,AZ|_{x}/\n{Z|_{x}}_g)}=\hat{c},
\]
that is, the sectional curvatures of all the two dimensional planes of $T|_x M$
that contain ${Z}|_x$ are equal to $\hat{c}$.
\end{proof}

The following examples show that there do exist
Riemannian manifolds, not both of constant curvature,
for which $\RDist$ in $Q$
has non-trivial (even nowhere vanishing) inner symmetries
of the type as described by the previous proposition.

\begin{example}
Suppose that $(M,g)$ is a Sasakian manifold of dimension $n$ with
a characteristic unit vector field $\xi$ (cf. \cite{boyer98}),
and suppose that $(\hat{M},\hat{g})$ is the $n$-dimensional unit sphere.
We show that $\LRD(\xi)$ is an inner symmetry of $\RDist$ on $Q$.

Indeed, we have for any $X,Y\in T|_x M$,
\[
AR(X\wedge \xi)Y=&A(g(\xi,Y)X-g(X,Y)\xi)
=\hat{g}(A\xi,AY)AX-\hat{g}(AX,AY)A\xi \\
=&(AX\wedge A\xi)AY=\hat{R}(AX\wedge A\xi)AY,
\]
where the last equality follows from the fact that $(\hat{M},\hat{g})$ has
constant curvature $=1$.
Thus for all $q=(x,\hat{x};A)\in Q$ and $X\in T|_x M$,
\[
\Rol_q(X\wedge \xi)=0.
\]
Setting $Z(q):=\xi|_x$, $\hat{Z}(q):=AZ(q)$ and $\ol{U}(q):=0$ for all $q=(x,\hat{x};A)\in Q$,
it follows from Proposition \ref{pr:inneq}
that $\LRD(\xi)=S_{(Z,\hat{Z},0)}\in\ISym(\RDist)$.

%Given any $q_0=(x_0,\hat{x}_0;A_0)\in Q$,
%there exists a unique unit Killing vector field $\hat{\xi}$ on $(\hat{M},\hat{g})$
%such that $\hat{\xi}|_{\hat{x}_0}=A_0\xi|_{x_0}$.
%This vector field $\hat{\xi}$ renders $(\hat{M},\hat{g})$ a Sasakian manifold
%with characteristic field $\hat{\xi}$.

\end{example}

\begin{example}
Let $K\in\R$ and suppose that $(\hat{M},\hat{g})$ is a space of constant curvature $K$.
Take as $(M,g)$ a warped product $(I\times N,\diff r^2+f(r)^2h)$
where $(N,h)$ is a Riemannian manifold of dimension $n-1$,
$I$ is a real interval, and $f:I\to\R$ is a strictly positive smooth function
that satisfies
\[
f''=-Kf.
\]

Denote by $\doo_r$ the canonical coordinate vector field on $I$.
We claim that defining for $q=(x,\hat{x};A)\in Q$,
$Z(q)=\doo_r|_{s}$ if $x=(s,y)$, $\hat{Z}(q)=A\doo_r|_s$ and $U(q)=0$,
then $S_{(Z,\hat{Z},0)}\in\ISym(\RDist)$.

Indeed, if $Y,Z\in T|_y N$, we have (see \cite{oneill83}, Chapter 7, Prop. 42)
\[
AR(Y\wedge \doo_r)\doo_r&=-\frac{f''}{f}AY=KAY
=K(AY\wedge A\doo_r)A\doo_r, \\
%=\hat{R}(AY\wedge A\doo_r)A\doo_r \\
AR(Y\wedge \doo_r)Z&=f''fh(Y,Z)\doo_r=-Kf^2h(Y,Z)\doo_r=-Kg(Y,Z)\doo_r=K(AY\wedge A\doo_r)AZ,
\]
which proves that $\Rol_q(X\wedge\doo_r)=0$ for all $q=(x,\hat{x};A)\in Q$ and $X\in T|_x M$.
Thus the claim follows again from Proposition \ref{pr:inneq}.
\end{example}

%%%%%%%%%%%%%%%%%%%%%%%%%%%%%%%%%%%%%
\section{Principal Bundle Structure}\label{sec:principal}
%%%%%%%%%%%%%%%%%%%%%%%%%%%%%%%%%%%%%

In this section, we state and prove one of the main results of the present paper: we give a necessary condition for the fiber bundle $\pi_{Q,M}\colon Q\to M$ to be a principal bundle. This characterization follows from a fundamental relation between the existence of certain symmetries and the group of Riemannian isometries. We use freely some classical results in Riemannian geometry, which can be found for example in~\cite{kobayashi63,sakai91}.

We introduce the convenient notation $S_{(\hat{Z},\ol{U})}:=S_{(0,\hat{Z},\ol{U})}$.
Moreover, we define
\[
\Sym_0(\RDist):=\{S\in \Sym(\RDist)\ |\ (\pi_{Q,M})_*S=0\},
\]
i.e., $S_{(Z,\hat{Z},\ol{U})}\in\Sym_0(\RDist)$ if and only if $Z=0$ and $S_{(0,\hat{Z},\ol{U})}\in\Sym(\RDist)$.

Observe that there is an equivalent characterization of the elements in $\Sym_0(\RDist)$, which follows easily from Proposition~\ref{pr:symeq}. We simply state it as a fact.

\begin{proposition}\label{pr:sym0eq}
$S_{(\hat{Z},\ol{U})}\in\Sym_0(\RDist)$ if and only if
\begin{align}
\ol{U}(q)X&=\LRD(X)|_q\hat{Z} \label{eq:sym0:1} \\
\LRD(X)|_q\ol{U}&=\hat{R}(AX\wedge \hat{Z}(q))A \label{eq:sym0:2},
\end{align}
for all $q=(x,\hat{x};A)\in Q$ and $X\in T|_x M$.
\end{proposition}

\vspace{0.5cm}

The following theorem gives a precise bound for the dimension of the vector space $\Sym_0(\RDist)|_{\mc{O}_{\RDist}(q_0)}$ of restrictions of elements of $\Sym_0(\RDist)$ onto an orbit $\mc{O}_{\RDist}(q_0)$.
Note the contrast with the space of inner symmetries since, as observed in Remark~\ref{rem:infdim}, if it is non trivial, then it is infinite dimensional.

\begin{theorem}\label{th:sym0bound}
Let $q_0=(x_0,\hat{x}_0;A_0)\in Q$.
Then the linear space $\Sym_0(\RDist)|_{\mc{O}_{\RDist}(q_0)}$
has dimension at most $\frac{n(n+1)}{2}$.
\end{theorem}

\begin{proof}
We claim that the map
\[
\Sym_0(\RDist)|_{\mc{O}_{\RDist}(q_0)}\to T|_{\hat{x}_0} \hat{M}\times \so(T|_{\hat{x}_0}\hat{M});
\quad S_{(\hat{Z},\ol{U})}|_{\mc{O}_{\RDist}(q_0)}\mapsto (\hat{Z}(q_0),A_0^{-1}\ol{U}(q_0))
\]
is injective. This will then imply
\[
\dim \big(\Sym_0(\RDist)|_{\mc{O}_{\RDist}(q_0)}\big)\leq \dim \big(T|_{\hat{x}_0} \hat{M}\times \so(T|_{\hat{x}_0}\hat{M})\big)=\frac{n(n+1)}{2},
\]
which is what we set out to prove.

Indeed, let $q_1=(x_1,\hat{x}_1;A_1)\in \mc{O}_{\RDist}(q_0)$ and
suppose that $\gamma:[0,1]\to M$ is a geodesic such that $\gamma(0)=x_1$ and $\dot{\gamma}(0)=X\in T|_{x_1} M$.
Write $q(t)=(\gamma(t),\hat{\gamma}(t);A(t))$ for the unique rolling curve in $Q$ starting at $q_1$ and satisfying $\pi_{Q,M}(q(t))=\gamma(t)$. %:=q_{\RDist}(\gamma,x_1)(t)$
Also write $\hat{Y}(t):=\hat{Z}(q(t))$.
Notice that 
\[
\LRD(\dot{\gamma}(t))|_{q(t)} \hat{Z}=\ol{\nabla}_{(\dot{\gamma}(t),\dot{\hat{\gamma}}(t))} (\hat{Z}(q(\cdot)))-\nu(\ol{\nabla}_{(\dot{\gamma}(t),\dot{\hat{\gamma}}(t))} A(\cdot))|_{q(t)} \hat{Z}
=\hat{\nabla}_{\dot{\hat{\gamma}}(t)} \hat{Y}(\cdot).
\]
Then by \eqref{eq:sym0:1}-\eqref{eq:sym0:2} one has
\[
\hat{R}(\dot{\hat{\gamma}}(t)\wedge \hat{Y}(t))\dot{\hat{\gamma}}(t)
=&\hat{R}(A(t)\dot{\gamma}(t)\wedge \hat{Z}(q(t)))A(t)\dot{\gamma}(t)
=\big(\LRD(\dot{\gamma}(t))|_{q(t)} \ol{U}(\cdot)\big)\dot{\gamma}(t) \\
=&\LRD(\dot{\gamma}(t))|_{q(t)} \big(\ol{U}(\cdot)\dot{\gamma}(\cdot)\big)-\ol{U}(q(t))\nabla_{\dot{\gamma}} \dot{\gamma} \\
=&\LRD(\dot{\gamma}(t))|_{q(t)} \big(\LRD(\dot{\gamma}(\cdot))|_{q(\cdot)} \hat{Z}(q(\cdot))\big) \\
=&\hat{\nabla}_{\hat{\gamma}(t)}\hat{\nabla}_{\hat{\gamma}(\cdot)} \hat{Y}(\cdot),
\]
where in the second to last equality we used that $\gamma$ is a geodesic.

Therefore, $\hat{Y}$ is a Jacobi field along the geodesic $\hat{\gamma}(t)=\hat{\gamma}_{A_1X}(t)$
and hence is uniquely determined by the initial values $\hat{Y}(0)=\hat{Z}(q_1)$,
$\hat{\nabla}_{A_1 X}\hat{Y}=\LRD(X)|_{q_1} \hat{Z}=\ol{U}(q_1)X$.
Moreover, $\hat{Y}$ uniquely determines $\ol{U}(q(t))$ for all $t$ since
\[
\LRD(\dot{\gamma}(t))|_{q(t)}\ol{U}=\ol{\nabla}_{(\dot{\gamma}(t),\dot{\hat{\gamma}}(t))}\ol{U}(q(\cdot))
-\nu(\ol{\nabla}_{(\dot{\gamma}(t),\dot{\hat{\gamma}}(t))} A(\cdot))|_{q(t)} \ol{U}
=\ol{\nabla}_{(\dot{\gamma}(t),\dot{\hat{\gamma}}(t))}\ol{U}(q(\cdot)),
\]
and hence
\[
\ol{U}(q(t))=P_0^t(\hat{\gamma})\Big(\ol{U}(q_1)+\int_0^t P_s^0(\hat{\gamma})\hat{R}(\dot{\hat{\gamma}}(s)\wedge \hat{Y}(s))P_0^s(\hat{\gamma})\diff s A_1\Big)P_t^0(\gamma).
\]
This implies that if $\hat{Z}(q_1)=\hat{Y}(q_1)$ and $\ol{U}(q_1)=\ol{V}(q_1)$,
then for all $X\in T|_{x_1} M$ and all $t$ one has
$S_{(\hat{Z},\ol{U})}|_{q(t)}=S_{(\hat{Y},\ol{V})}|_{q(t)}$,
where $q(t)=q_{\RDist}(\gamma_X,q_1)$ and $\gamma_X$ is the geodesic with $\gamma_X(0)=x_1$, $\dot{\gamma}_X(0)=X$.

To finish the proof,
suppose $(\hat{Z}(q_0),A_0^{-1}\ol{U}(q_0))=(\hat{Y}(q_0),A_0^{-1}\ol{V}(q_0))$.
Given a point $q=(x,\hat{x};A)\in\mc{O}_{\RDist}(q_0)$,
there exists geodesics $\gamma_i:[0,1]\to M$, $i=1,\dots,N$, on $M$ such that $\gamma_1(0)=x_0$,
$\gamma_N(1)=x$, $\gamma_{i-1}(1)=\gamma_i(0)$, $i=2,\dots,N$
and $q_{\RDist}(\gamma_N\ldots\gamma_2.\gamma_1,q_0)(1)=q$.
Since $\hat{Z}(q_0)=\hat{Y}(q_0)$, $\ol{U}(q_0)=\ol{V}(q_0)$,
we have $S_{(\hat{Z},\hat{U})}|_{q_{\RDist}(\gamma_1,q_0)(t)}=S_{(\hat{Y},\hat{V})}|_{q_{\RDist}(\gamma_1,q_0)(t)}$,
$t\in [0,1]$,
by what we just proved above.
In particular, $\hat{Z}(q_1)=\hat{Y}(q_1)$, $\ol{U}(q_1)=\ol{V}(q_1)$,
where $q_i:=q_{\RDist}(\gamma_i\ldots\gamma_2.\gamma_1,q_0)=q_{\RDist}(\gamma_i,q_{i-1})$.
Inductively we obtain $\hat{Z}(q_N)=\hat{Y}(q_N)$, $\ol{U}(q_N)=\ol{V}(q_N)$,
where $q_N=q$, and so $S_{(\hat{Z},\ol{U})}|_q=S_{(\hat{Y},\ol{V})}|_q$.
Since $q\in \mc{O}_{\RDist}(q_0)$ was arbitrary,
we have proven the claim.
\end{proof}

\begin{remark}
The proof of the previous theorem shows in fact
that for any $q_0\in Q$ the space $\Sym_0(\RDist|_{\mc{O}_{\RDist}(q_0)})$
of all $S\in \Sym(\RDist|_{\mc{O}_{\RDist}(q_0)})$ such that $(\pi_{Q,M}|_{\mc{O}_{\RDist}(q_0)})_*S=0$
has at most dimension $n(n+1)/2$.
\end{remark}

The theorem above has a very natural consequence in the case of a completely controllable rolling dynamics. Recall that the rolling distribution $\RDist$ is said to be completely controllable if $\mc{O}_{\RDist}(q_0)=Q$, $q_0\in Q$.

\begin{corollary}\label{cor:sym0bound}
Assuming that $\RDist$ is completely controllable, then $\Sym_0(\RDist)$ is at most of dimension $\frac{n(n+1)}{2}$.
\end{corollary}

The next proposition provides a lower bound for the dimension of the space $\Sym_0(\RDist)$ in terms of Riemannian isometries of $(\hM,\hat g)$. 

\begin{proposition}\label{pr:sym0lowbound}
The dimension of the space $\Sym_0(\RDist)$ is at least $\dim\mathrm{Kil}(\hat{M},\hat{g})$,
where $\mathrm{Kil}(\hat{M},\hat{g})$ is the Lie-algebra of Killing fields of $(\hat{M},\hat{g})$.
\end{proposition}

\begin{proof}
Let $\hat{K}$ be a Killing field on $(\hat{M},\hat{g})$.
Defining $\hat{Z}(q):=\hat{K}|_{\hat{x}}$, $\ol{U}(q):=\hat{\nabla}\hat{K}|_{\hat{x}} A$,
for $q=(x,\hat{x};A)\in Q$
and recalling that $\hat{K}$ satisfies
\[
\hat{\nabla}_{\hat{X}}(\hat{\nabla}\hat{K})=\hat{R}(\hat{X}\wedge \hat{K}),\quad \forall \hat{X}\in T\hat{M},
\]
we see that Eqs. \eqref{eq:sym0:1}-\eqref{eq:sym0:2} are satisfied:
\[
& \LRD(X)|_q\hat{Z}=\hat{\nabla}_{AX}\hat{K}=(\hat{\nabla}\hat{K})AX=\ol{U}(q)X \\
& \LRD(X)|_q\ol{U}=(\LRD(X)|_q \hat{\nabla}\hat{K})A=\hat{\nabla}_{AX}(\hat{\nabla}\hat{K})A
=\hat{R}(AX\wedge \hat{K})A=\hat{R}(AX\wedge \hat{Z}(q))A,
\]
i.e. $S_{(\hat{Z},\ol{U})}\in\Sym_0(\RDist)$.
Therefore, each Killing field of $(\hat{M},\hat{g})$ determines
a unique element of $\Sym_0(\RDist)$,
and this implies the claim.
\end{proof}

Using the notations above, we present a technical lemma to identify Killing vector fields on $\hM$ via symmetries of the rolling model.

\begin{lemma}\label{le:sym0iso}
Suppose that $S_{(\hat{Z},\ol{U})}\in\Sym_0(\RDist)$.
If there is a vector field $\hat{K}\in\VF(\hat{M})$ such that $\hat{Z}(q)=\hat{K}|_{\hat{x}}$
for all $q=(x,\hat{x};A)$, then $\hat{K}$ is a Killing field on $(\hat{M},\hat{g})$.
Moreover, if this is the case, $\ol{U}(q)=\hat{\nabla}\hat{K}|_{\hat{x}}A$
for all $q=(x,\hat{x};A)\in Q$.
\end{lemma}

\begin{proof}
If there exists $\hat{K}\in\VF(\hat{M})$ such that $\hat{Z}(q)=\hat{K}|_{\hat{x}}$ for all $q=(x,\hat{x};A)\in Q$,
then for all $X\in T|_x M$
\[
\ol{U}(q)X=\LRD(X)|_q\hat{Z}=\hat{\nabla}_{AX} \hat{K}.
\]
Writing $\ol{U}(q)=\hat{U}(q)A$, where $\hat{U}\in\Cinf(\pi_{Q,\hat{M}},\pi_{\so(T\hat{M})})$,
we get
\[
\hat{\nabla}_{\hat{X}}\hat{K}=\hat{U}(q)\hat{X},
\]
for all $\hat{X}\in T|_{\hat{x}}\hat{M}$. Since $\hat{U}(q)\in\so(T|_{\hat{x}}\hat{M})$,
one sees that $\hat{\nabla}\hat{K}|_{\hat{x}}$ is a skew-symmetric map,
and since $\hat{x}\in\hat{M}$ was arbitrary, it follows that $\hat{K}$ is a Killing field.
Moreover, $\ol{U}(q)=\hat{U}(q)A=\hat{\nabla}\hat{K}|_{\hat{x}}A$.
\end{proof}

From now on, $O$ will be an open subset of $Q$ and $S_{(Z,\ol{U})}\in\Sym(\RDist|_O)$.
In studying $S_{(\hat{Z},\ol{U})}\in\Sym_0(\RDist|_O)$,
we write from now on
\[
\ol{U}(q)=\hat{U}(q)A=AU(q),
\]
where $U\in\Cinf(\pi_{O,M},\pi_{\so(TM)})$, $\hat{U}\in \Cinf(\pi_{O,\hat{M}},\pi_{\so(T\hat{M})})$.
As before $\pi_{O,M}=\pi_{Q,M}|_O$ and $\pi_{O,\hat{M}}=\pi_{Q,\hat{M}}|_O$.

The next two propositions reveal some interesting phenomena that occurs when we assume that the rolling curvature map is invertible. They form a core technical part of the proof of Theorem~\ref{th:sym0}, which is the main result of this section.

\begin{proposition}\label{pr:nuU_LZ}
Suppose that $O$ is an open subset of $Q$ such that for all $q=(x,\hat{x};A)\in O$, the map
$\wRol_q:\bigwedge^2 T|_x M\to \bigwedge^2 T|_x M$ is invertible.
Then for all $\hat{C}\in\so(T|_{\hat{x}} \hat{M})$ and $\hat{Y}\in T|_{\hat{x}} \hat{M}$, where $q=(x,\hat{x};A)\in O$, one has
\[
\big(\nu(\hat{C}A)|_q\hat{U}\big)\hat{Y}=\LNSD(0,\hat{C}\hat{Y})|_q\hat{Z}-\hat{U}(q)\hat{C}\hat{Y}.
\]
Equivalently, for all $C\in\so(T|_x M)$, $W\in T|_x M$,
\[
\big(\nu(AC)|_q\hat{U}\big)AW=-\LNSD(CW,0)|_q\hat{Z}.
\]
\end{proposition}

\begin{proof}
Notice that by Proposition \ref{pr:vertZ}, $\nu(\Rol_q(\xi))|_q \hat{Z}=0$
for all $\xi\in \bigwedge^2 TM$.
By assumption, $\wRol_q$ is invertible
and hence, for any $\hat{C}\in\so(T|_{\hat{x}} \hat{M})$,
there exists a $\xi$ such that $\wRol_q(\xi)=A^{-1}\hat{C}A$, i.e., $\hat{C}A=\Rol_q(\xi)$
and hence $\nu(\hat{C}A)|_q\hat{Z}=0$ by Proposition \ref{pr:vertZ}.
Recall that the vectors $\nu(\hat{C}A)|_q$, with $\hat{C}\in\so(T|_{\hat{x}}\hat{M})$,
span $V|_q(\pi_Q)$.
 
For all $X,Y,W\in\VF(M)$ one has
\[
(\nu(\Rol_q(X\wedge Y))|_q \ol{U})W&=\nu(\Rol_q(X\wedge Y))|_q (\ol{U}W)
=\nu(\Rol_q(X\wedge Y))|_q(\LRD(W)\hat{Z}) \\
&=\LRD(W)|_q\big(\underbrace{\nu(\Rol(X\wedge Y))\hat{Z}}_{=0}\big)+[\nu(\Rol(X\wedge Y)),\LRD(W)]|_q \hat{Z} \\
&=\LNSD(\Rol_q(X\wedge Y)W)|_q\hat{Z}-\underbrace{\nu\big(\LRD(W)|_q\big(\Rol(X\wedge Y)\big)\big)\big|_q \hat{Z}}_{=0}
\]
Appealing again to the invertibility of $\wRol_q$,
this implies that for any $\hat{C}\in \so(T|_{\hat{x}}\hat{M})$, we have
\[
(\nu(\hat{C}A)|_q \ol{U})W=\LNSD(\hat{C}AW)|_q\hat{Z}.
\]
Finally, since $\ol{U}(q)=\hat{U}(q)A$, we have
\[
\nu(\hat{C}A)|_q \ol{U}=\hat{U}(q)\hat{C}A+(\nu(\hat{C}A)|_q\hat{U})A,
\]
and so
\[
\hat{U}(q)\hat{C}AW+(\nu(\hat{C}A)|_q\hat{U})AW=\LNSD(\hat{C}AW)|_q\hat{Z}.
\]
Since $W$ was arbitrary, we may replace $AW$ by an arbitrary vector $\hat{Y}\in T|_{\hat{x}}\hat{M}$.

This implies that, for all $W\in T|_x M$, $C\in\so(T|_x M)$
(notice that $A\so(T|_x M)=\so(T|_{\hat{x}}\hat{M})A$)
\[
\big(\nu(AC)|_q\hat{U}\big)AW=&\LNSD(ACW)|_q\hat{Z}-\hat{U}(q)ACW \\
&=\LRD(CW)|_q\hat{Z}-\LNSD(CW,0)|_q\hat{Z}-\hat{U}(q)ACW \\
&=\hat{U}(q)ACW-\LNSD(CW,0)|_q\hat{Z}-\hat{U}(q)ACW \\
&=-\LNSD(CW,0)|_q\hat{Z}.
\]
This completes the proof.
\end{proof}

\begin{proposition}\label{pr:sym0alg}
Assume that $\wRol_q$ is invertible at every $q\in O$.
Then for all $q=(x,\hat{x};A)\in O$, $C\in \so(T|_x M)$ and $X,W\in T|_x M$ one has
\[
(\LNSD(CW,0)|_q\hat{U})AX+(\LNSD(CX,0)|_q\hat{U})AW=0.
\]
\end{proposition}

\begin{proof}
Let $C\in\Gamma(\pi_{\so(TM)})$ and $X,W\in\VF(M)$
and define $Y\in\VF(M)$ by $Y|_x:=C|_x W|_x$ for all $x\in M$.
Then
\[
\LNSD(Y,0)|_q(\ol{U}(\cdot)X)
&=\LNSD(Y,0)|_q(\LRD(X)\hat{Z}) \\
&=[\LNSD(Y,0),\LRD(X)]|_q\hat{Z}+\LRD(X)|_q(\LNSD(Y,0)\hat{Z}) \\
&=\LRD(\nabla_Y X)|_q\hat{Z}-\LNSD(\nabla_X Y,0)|_q\hat{Z}+\nu(AR(Y,X))|_q\hat{Z} \\
&+R^{\ol{\nabla}}((Y,0),(X,AX))\hat{Z}(q)
+\LRD(X)|_q(\LNSD(Y,0)\hat{Z}).
\]
%From Proposition \ref{pr:vertZ}, $\nu(AR(Y,X))|_q\hat{Z}=0$.
By the first part of the proof of the previous proposition, $\nu(AR(Y,X))|_q\hat{Z}=0$.
It is also clear that $R^{\ol{\nabla}}((Y,0),(X,AX))\hat{Z}(q)=0$.
Hence
\[
\LNSD(Y,0)|_q(\ol{U}(\cdot)X)=\hat{U}(q)A\nabla_Y X-\LNSD(\nabla_X Y,0)|_q\hat{Z}+\LRD(X)|_q(\LNSD(Y,0)\hat{Z}).
\]

One has by Proposition \ref{pr:nuU_LZ},
\begin{align*}
\LRD(X)|_q(\LNSD(Y,0)\hat{Z})&
=\LRD(X)|_q(\LNSD(CW,0)\hat{Z})
=-\LRD(X)|_q\big((\nu((\cdot)C)\hat{U})(\cdot)W\big) \\
&=-(\nu(AC)|_q\hat{U})A\nabla_X W
-\big([\LRD(X),\nu((\cdot)C)]|_q\hat{U}\big)AW\\
&-\big(\nu(AC)|_q(\LRD(X)\hat{U})\big)AW \\
&=\LNSD(C\nabla_X W,0)|_q\hat{Z}-(\nu(A\nabla_X C)|_q\hat{U})AW\\
&+(\LNSD(0,ACX)|_q\hat{U})AW
-\big(\nu(AC)|_q\big(\hat{R}((\cdot)X\wedge \hat{Z})\big)\big) AW\\
&=\LNSD(C\nabla_X W+(\nabla_X C)W,0)|_q\hat{Z}+(\LNSD(0,ACX)|_q\hat{U})AW \\
&-\hat{R}(ACX\wedge\hat{Z}(q))AW-\hat{R}(AX\wedge \nu(AC)|_q\hat{Z})AW.
\end{align*}
Since $C\nabla_X W+(\nabla_X C)W=\nabla_X Y$ and since $ \nu(AC)|_q\hat{Z}=0$, this simplifies to
\[
\LRD(X)|_q(\LNSD(Y,0)\hat{Z})&=\LNSD(\nabla_X Y,0)|_q\hat{Z}\\
&+(\LNSD(0,ACX)|_q\hat{U})AW-\hat{R}(ACX\wedge\hat{Z}(q))AW.
\]
Since, on the other hand,
\[
\LNSD(Y,0)|_q(\ol{U}(\cdot)X)=(\LNSD(Y,0)|_q\hat{U})AX+\hat{U}(q)A\nabla_Y X,
\]
we have arrived at
\[
(\LNSD(Y,0)|_q\hat{U})AX+\hat{U}(q)A\nabla_Y X
&=\LNSD(Y,0)|_q(\ol{U}(\cdot)X) \\
&=\hat{U}(q)A\nabla_Y X-\LNSD(\nabla_X Y,0)|_q\hat{Z}
+\LNSD(\nabla_X Y,0)|_q\hat{Z}\\
&+(\LNSD(0,ACX)|_q\hat{U})AW-\hat{R}(ACX\wedge\hat{Z}(q)),
\]
that is
\[
(\LNSD(Y,0)|_q\hat{U})AX=(\LNSD(0,ACX)|_q\hat{U})AW-\hat{R}(ACX\wedge\hat{Z}(q))AW.
\]
Finally, using
\[
(\LNSD(0,ACX)|_q\hat{U})AW
&=(\LRD(CX)|_q\hat{U})AW-(\LNSD(CX,0)|_q\hat{U})AW \\
&=\hat{R}(ACX\wedge \hat{Z}(q))AW-(\LNSD(CX,0)|_q\hat{U})AW,
\]
and recalling that $Y=CW$, we obtain
\begin{equation*}
(\LNSD(CW,0)|_q\hat{U})AX=-(\LNSD(CX,0)|_q\hat{U})AW.\qedhere
\end{equation*}
\end{proof}

%\begin{corollary}
%Assuming that $\Rol$ is invertible on $Q$,
%%then there exists $\hat{W}\in\Cinf(\pi_{Q,T\hat{M}},\pi_{T\hat{M}})$
%%such that for all $q=(x,\hat{x};A)\in Q$ and $X\in T|_x M$
%then for all $q=(x,\hat{x};A)\in Q$ and $X,Y\in T|_x M$ such that $X\perp Y$,
%one has
%\[
%%\LNSD(X,0)|_q \hat{U}=\hat{g}(AX,\cdot)\hat{W}(q).
%(\LNSD(X,0)|_q \hat{U})AY=0.
%\]
%\end{corollary}
%
%\begin{proof}

%Let $\hat{V}\in T|_x M$. Writing $\hat{V}=\alpha AX+AY$, where $Y\perp X$ (i.e. $AY\perp AX$), we get
%\[
%(\LNSD(X,0)|_q\hat{U})\hat{V}=\alpha (\LNSD(X,0)|_q\hat{U})AX.
%\]
%Defining $\hat{W}(q):=(\LNSD(X,0)|_q\hat{U})AX$ and observing that
%\[
%\alpha\hat{g}(AX,AX)=\hat{g}(\hat{V},AX),
%\]
%we get
%\[
%(\LNSD(X,0)|_q\hat{U})\hat{V}=\hat{g}(AX,\hat{V})\hat{W}(q)
%\]
%which concludes the proof.
%\end{proof}

\begin{corollary}\label{cor:sym0alg}
Assuming that $\wRol$ is invertible on $O$,
then for all $q=(x,\hat{x};A)\in O$ and $X\in T|_x M$,
one has
\[
\LNSD(X,0)|_q\hat{U}=0.
\]
\end{corollary}

\begin{proof}
We may well assume that $X$ has unit length.
Let $Y\in T|_x M$ be a unit vector such that $Y\perp X$.
By the previous proposition, for any $W\in T|_x M$ and $C\in\so(T|_x M)$, we have
\[
(\LNSD(CW,0)|_q\hat{U})AW=0.
\]
Therefore, choosing any $C\in\so(T|_x M)$ such that $CY=X$
(one can take e.g. $C=X\wedge Y$),
%(such a $C$ exists since $X\perp Y$),
we have
\begin{align}\label{eq:LXY}
0=(\LNSD(CY,0)|_q\hat{U})AY=(\LNSD(X,0)|_q\hat{U})AY.
\end{align}

Since $\LNSD(X,0)|_q\hat{U}\in\so(T|_{\hat{x}} \hat{M)}$,
\[
\hat{g}((\LNSD(X,0)|_q\hat{U})AX,AX)=0.
\]
On the other hand, since $\LNSD(X,0)|_q\hat{U}\in\so(T|_{\hat{x}} \hat{M)}$ again and by Eq. \eqref{eq:LXY} ,
%if $Y\in T|_x M$ is such that $Y\perp X$, then,
\[
\hat{g}((\LNSD(X,0)|_q\hat{U})AX,AY)=-\hat{g}(AX,\underbrace{(\LNSD(X,0)|_q\hat{U})AY}_{=0})=0.
\]
%where we used the facts that $\LNSD(X,0)|_q\hat{U}\in\so(T|_{\hat{x}} \hat{M)}$
%and the previous corollary.
The above shows that
\begin{align}\label{eq:LXX}
(\LNSD(X,0)|_q\hat{U})AX=0.
\end{align}

Finally, if $\hat{V}\in T|_{\hat{x}}\hat{M}$, it can be written as
$\hat{V}=\alpha AX+\beta AY$ where $\n{Y}_g=1$, $X\perp Y$
and therefore Eqs. \eqref{eq:LXY}, \eqref{eq:LXX} imply that
\begin{equation*}
(\LNSD(X,0)|_q\hat{U})\hat{V}
=\alpha(\LNSD(X,0)|_q\hat{U})AX+\beta (\LNSD(X,0)|_q\hat{U})AY=0.\qedhere
\end{equation*}
\end{proof}

After all these preparatory propositions and lemmas, we can state the main result of this section. 

\begin{theorem}\label{th:sym0}
If there is an open dense set $O\subset Q$
such that $R|_x:\bigwedge^2 T|_x M\to \bigwedge^2 T|_x M$ is invertible on $\pi_{Q,M}(O)$
and $\wRol$ is invertible on $O$,
%and $\widetilde{\Rol}_q:\bigwedge^2 T|_x M\to \bigwedge^2 T|_x M$ is invertible on $O$,
then, up to an isomorphism of Lie-algebras,
\[
\Sym_0(\RDist)=\mathrm{Iso}(\hat{M},\hat{g})
\]
and therefore all the elements of $\Sym_0(\RDist)$ are induced by Killing fields of $(\hat{M},\hat{g})$.

In particular, under the above assumptions, if
there is a principal bundle structure on $\pi_{Q,M}:Q\to M$
that renders $\RDist$ to a principal bundle connection,
then $(\hat{M},\hat{g})$ is a space of constant curvature.

\end{theorem}

\begin{proof}
Let $S=S_{(\hat{Z},\hat{U})}\in \Sym_0(\RDist)$ be given.
For any $q\in O$ and any $X,Y\in\VF(M)$, we have by Corollary \ref{cor:sym0alg},
\[
0&=\LNSD(X,0)|_q(\LNSD(Y,0)\hat{U})
=\underbrace{\LNSD(Y,0)|_q(\LNSD(X,0)\hat{U})}_{=0}
+[\LNSD(X,0),\LNSD(Y,0)]|_q\hat{U} \\
&=\underbrace{\LNSD([X,Y],0)|_q\hat{U}}_{=0}+\nu(AR(X,Y))|_q\hat{U}+\underbrace{R^{\ol{\nabla}}((X,0),(Y,0))\hat{U}(q)}_{=0} \\
&=\nu(AR(X,Y))|_q\hat{U}.
\]

Given any $C\in \so(T|_x M)$. Since $R|_x$ is invertible,
there exists a $\xi\in\bigwedge^2 T|_x M$ such that $R(\xi)=C$
and therefore the above shows that
\[
\nu(AC)|_q\hat{U}=\nu(AR(\xi))|_q\hat{U}=0.
\]
Hence, Proposition \ref{pr:nuU_LZ} implies that for all $X\in T|_x M$,
\[
\LNSD(CX,0)|_q\hat{Z}=-(\nu(AC)|_q\hat{U})AX=0.
\]
Since $C,X$ were arbitrary, we have that 
\[
\LNSD(X,0)|_q\hat{Z}=0,\quad \forall q=(x,\hat{x};A)\in O,\ \forall X\in T|_x M.
\]

By the above and Proposition \ref{pr:vertZ},
we have that for all $q=(x,\hat{x};A)\in O$, $X\in T|_x M$, $\hat{U}\in \so(T|_{\hat{x}}\hat{M})$,
\begin{align}\label{eq:DZ0}
(\LNSD(X,0)|_q+\nu(\hat{U}A)|_q)\hat{Z}=0.
\end{align}
By density of $O$ in $Q$,
a continuity argument implies that this holds for all $q=(x,\hat{x};A)\in Q$, $X\in T|_x M$ and $\hat{U}\in \so(T|_{\hat{x}}\hat{M})$.

%%Clearly, $\hat{U}:=\pi_{Q,\hat{M}}(O)$ is an open dense subset of $\hat{M}$.
%Let $\hat{x}_0\in \hat{M}$ be given and
%let $q_0$ be a point in $Q$ such that $\pi_{Q,\hat{M}}(q_0)=\hat{x}_0$,
%say $q_0=(x_0,\hat{x}_0;A_0)$.
%%Choose a neighborhood $O'$ of $q_0$ such that $O'\subset O$
%%and $\pi_{Q,\hat{M}}^{-1}(\hat{x})\cap O'$ is connected
%%for all $\hat{x}\in \pi_{Q,\hat{M}}(O')=:\hat{U}'$.
Given $\hat{x}\in \hat{M}$,
choose $q_1,q_2\in \pi_{Q,\hat{M}}^{-1}(\hat{x})$,
say $q_1=(x_1,\hat{x};A_1)$, $q_1=(x_2,\hat{x};A_2)$.
Choose a path $\Gamma$ in $\pi_{Q,\hat{M}}^{-1}(\hat{x})$
(recall that $\pi_{Q,\hat{M}}^{-1}(\hat{x})$ is connected
since it is diffeomorphic to the oriented orthonormal frame bundle of $M$
which is connected)
such that $\Gamma(0)=q_1$, $\Gamma(1)=q_2$.
Since $\Gamma(t)\in \pi_{Q,\hat{M}}^{-1}(\hat{x})$,
we have $\dot{\Gamma}(t)\in T\pi_{Q,\hat{M}}^{-1}(\hat{x})$ for all $t$
and so
\[
\dot{\Gamma}(t)=\LNSD(X(t),0)|_{\Gamma(t)}+\nu(\hat{U}(t)A)|_{\Gamma(t)},
\]
where $X(t)\in TM$, $\hat{U}(t)\in\so(T\hat{M})$.
Therefore, by Eq. \eqref{eq:DZ0},
\[
\dot{\Gamma}(t)\hat{Z}=0,\quad \forall t\in [0,1].
\]

We claim that $\hat{Z}|_{q_1}=\hat{Z}|_{q_2}$.
Indeed, given $\hat{f}\in\Cinf(\hat{M})$,
we can compute, by considering $\hat{Z}\hat{f}$ as the function $q'=(x',\hat{x}';A')\mapsto \hat{Z}(q')\hat{f}(x')=\hat{Z}(q')\pi_{Q,\hat{M}}^*\hat{f}(q')$ on $Q$,
(we write, for the sake of clarity, $\ctr$ for the contraction of tensors)
\[
&\dif{t}(\hat{Z}|_{\Gamma(t)}\hat{f})
=\dot{\Gamma}(t)(\hat{Z}\hat{f})
=\dot{\Gamma}(t)(\hat{Z}\pi_{Q,\hat{M}}^*\hat{f})
=\dot{\Gamma}(t)(\hat{Z}\ctr\diff\pi_{Q,\hat{M}}^*\hat{f})
=\dot{\Gamma}(t)(\hat{Z}\ctr\pi_{Q,\hat{M}}^*\diff\hat{f}) \\
=&(\dot{\Gamma}(t)\hat{Z})\ctr \pi_{Q,\hat{M}}^*\diff\hat{f}|_{\Gamma(t)}
+\hat{Z}|_{\Gamma(t)}\ctr \dot{\Gamma}(t)\pi_{Q,\hat{M}}^*\diff\hat{f}
\]
Now obviously $\dot{\Gamma}(t)\pi_{Q,\hat{M}}^*\diff\hat{f}=0$
and we also know that $\dot{\Gamma}(t)\hat{Z}=0$.
Therefore, $\dif{t}(\hat{Z}|_{\Gamma(t)}\hat{f})=0$ for all $t\in [0,1]$
and thus
$\hat{Z}|_{\Gamma(0)}\hat{f}=\hat{Z}|_{\Gamma(1)}\hat{f}$.
Since $\hat{f}\in\Cinf(\hat{M})$ was arbitrary,
and since $\Gamma(0)=q_1$, $\Gamma(1)=q_2$,
the claim follows. 

Hence we have shown that
for all $\hat{x}\in\hat{M}$ and all $q_1,q_1\in \pi_{Q,\hat{M}}^{-1}(\hat{x})$,
$\hat{Z}|_{q_1}=\hat{Z}|_{q_2}$.
Thus, defining
\[
\hat{Y}_S|_{\hat{x}}:=\{\hat{Z}|_q\ |\ q\in  \pi_{Q,\hat{M}}^{-1}(\hat{x})\},\quad \hat{x}\in\hat{M}
\]
we have shown that $\hat{Y}_S|_{\hat{x}}$ is a singleton set
for every $\hat{x}\in\hat{M}$
and therefore it defines a map $\hat{M}\to T\hat{M}$,
which we write as $\hat{Y}_S$ as well,
such that $\hat{Y}_S|_{\hat{x}}\in T|_{\hat{x}}\hat{M}$, $\forall \hat{x}\in\hat{M}$.
By using smooth local sections of $\pi_{Q,\hat{M}}$,
the smoothness of $\hat{Y}_S$ follows from that of $\hat{Z}$,
i.e., $\hat{Y}_S$ is a vector field on $\hat{M}$.

It follows from Lemma \ref{le:sym0iso} that $\hat{Y}_S$
is a Killing field on $(\hat{M},\hat{g})$.
It is clear that the map $S\mapsto \hat{Y}_S$
from $\Sym_0(\RDist)$ into $\mathrm{Kil}(\hat{M},\hat{g})$, the space of Killing fields of $(\hat{M},\hat{g})$,
is injective and therefore,
\[
\dim \Sym_0(\RDist)\leq \dim\mathrm{Iso}(\hat{M},\hat{g}).
\]
Proposition \ref{pr:sym0lowbound} provides the opposite inequality, thus
we have completed the first part of the proof.

To prove the last claim of the theorem,
suppose that $\mu:G\times Q\to Q$ is a principal bundle structure ($G$ is a Lie group)
on $\pi_{Q,M}$ such that $\mu_*\RDist=\RDist$.
Then for every $X\in\mathfrak{g}$,
the vector field defined by $S_X|_q:=(\mu^q)_*X$
belongs to $\Sym_0(\RDist)$
(notice that $(\pi_{Q,M})_*S_X|_q=(\pi_{Q,M}\circ \mu^q)_*X=0$,
since $(\pi_{Q,M}\circ\mu^q)(a)=\pi_{Q,M}(q)$ for all $a\in G$).
The fact that $\dim\pi_{Q,M}^{-1}(x)=\frac{n(n+1)}{2}$
for all $x\in M$ implies that $\dim G=\frac{n(n+1)}{2}$.
Given a basis $X_i$, $i=1,\dots,\frac{n(n+1)}{2}$, of $\mathfrak{g}$,
we have that $S_{X_i}\in\Sym_0(\RDist)$
are linearly independent
and hence so are the Killing fields $\hat{Y}_{S_{X_i}}$
of $(\hat{M},\hat{g})$.
This implies that $\dim\mathrm{Kil}(\hat{M},\hat{g})\geq \frac{n(n+1)}{2}$
and because $\frac{n(n+1)}{2}$ is the maximal dimension of $\mathrm{Kil}(\hat{M},\hat{g})$,
we have an equality.
But this implies, by a well known theorem in Riemannian geometry (see \cite{kobayashi63})
that $(\hat{M},\hat{g})$ must have constant curvature.
\qedhere

\end{proof}

\section{Constant Curvature and Flatness}\label{sec:flat}
%%%%%%%%%%%%%%%%%%%%%%%%%%%

The aim of this section is to show the remaining main result of this paper, which concerns the impossibility for the rolling distribution for spaces of constant sectional curvature rolling to be flat, when the dimensions are greater than or equal to three. Recall our assumption that $(M,g)$, $(\hat{M},\hat{g})$ are spaces of constant curvatures, $K$ and $\hat{K}$, respectively.

%%%%%%%%%%%%%%%%%%%%%%%%%%%
\subsection{Nilpotent Approximation of the rolling distribution $\RDist$}
%%%%%%%%%%%%%%%%%%%%%%%%%%%

Define $\kappa:=-K+\hat{K}\neq 0$.
Let $X,Y,Z\in\VF(M)$.
Then, after standard computations (see for instance \cite{CK})
\[
[\LRD(Y),\LRD(Z)]
&=\LRD([Y,Z])+\kappa \nu(A(Y\wedge Z)) \\
[\LRD(X),[\LRD(Y),\LRD(Z)]]
&=-\kappa\LNSD(A(Y\wedge Z)X)+\LRD([X,[Y,Z]]) \\
&+\nu(A(X\wedge [Y,Z])+\kappa §A\LRD(X)(Y\wedge Z)) \\
&=-\kappa g(Z,X)\LNSD(Y)+\kappa g(Y,X)\LNSD(Z)\qquad\mathrm{mod}(\LRD,\nu).
\]
Recall that $\nu(A(X_i\wedge X_j))$, with $1\leq i<j\leq n$ form a basis of the vertical fiber $\pi_Q^{-1}(x,\hat{x})$ at $q=(x,\hat{x};A)$, which is of dimension $\frac{n(n-1)}2$. Since $\kappa\neq 0$, one easily gets that the vertical part of the second order Lie brackets generate the full vertical fiber at every $q=(x,\hat{x},A)\in Q$ and with the third order Lie brackets, one gets the $n$ directions generating 
$T|_{\hat{x}}\hat{M}$. Therefore, the control system defined by $\RDist$ is completely controllable and equiregular, i.e. the growth vector of the distribution $\RDist$  at every point $q\in Q$ is equal to $(n,\frac{n(n+1)}2,2n+\frac{n(n-1)}2)$. One then gets that $\RDist$ defines a sub-Riemannian structure on $Q$ and the structure of (isometric) nilpotent approximations at every point $q\in Q$ are given next. 

%\remember{Why mention that $\RDist$ "then" defines a SR-structure?}

\begin{proposition}\label{nilpo}
Set $m:=2n+\frac{n(n-1)}2$. For every point $q\in Q$, any nilpotent approximation of $\RDist$ at $q$ is given by an $n$-dimensional distribution $D$ in $\mathbb{R}^m$ admitting a global basis of vector fields $N_1,\cdots,N_n$ such that $\mathrm{Lie}(D)$, the Lie algebra generated by the $N_i'$s is a graded nilpotent Lie algebra of step $3$ so that
\begin{equation}\label{decomp}
\mathbb{R}^m=\mathrm{Lie}(D)=\mathfrak{n}_1\oplus\mathfrak{n}_2\oplus\mathfrak{n}_3,
\end{equation}
where $\mathfrak{n}_1=D$, $\mathfrak{n}_2=[D,D]$ with $\dim\mathfrak{n}_2=\frac{n(n-1)}2$
and $\mathfrak{n}_3$ is $n$-dimensional and admits a basis $Z_1,\cdots,Z_n$ such that, 
 for $1\leq i, j,k\leq n$, 
\[
[N_i,[N_j,N_k]]=-\delta_{ik}Z_j+\delta_{ij}Z_k.
\]
\end{proposition}
\begin{proof}
The definition of a nilpotent approximation is given in \cite{bella}. By standard computations, cf. \cite{bella}, one gets that $\mathrm{Lie}(D)$ is actually isomorphic to $F_{n,3}/Z$, where $F_{n,3}$
is the free Lie algebra of step $3$ with $n$ generators $X_1,\cdots,X_n$  and $Z$ is the involutive Lie algebra spanned by $[N_i,[N_j,N_k]]+\delta_{ik}[X_i,[X_i,X_j]]-\delta_{ij}[X_i,[X_i,X_k]]$, with $1\leq i, j,k\leq n$.
\end{proof}
%
%The nilpotent approximation of the rolling distribution $\RDist$ at every point $q\in Q$ is given as the nilpotent Lie algebra $\mathfrak{n}$ defined by 
%$\mathfrak{n}=\mathfrak{n}_1\oplus\mathfrak{n}_2\oplus\mathfrak{n}_3$
%such that $\mathfrak{n}_1$ has basis $X_1,\dots,X_n$,
%$\mathfrak{n}_2$ has basis $[X_i,X_j]$, for $1\leq i< j\leq n$, and $\mathfrak{n}_3$ has basis $Z_1,\dots,Z_n$ verifying, for $1\leq i\neq j\leq n$, 
%\[
%[X_i,[X_i,X_j]]=Z_j
%\]
%and, more genrally, for $1\leq i, j,k\leq n$, 
%\[
%[X_i,[X_j,X_k]]=-\delta_{ik}Z_j+\delta_{ij}Z_k.
%\]
%Lie brackets of the $X_i$'s of order greater than or equal to four are null.
%%\end{proposition}
%%Moreover, $X_{ij}:=[X_i,X_j]$, $i<j$, are linearly independent.
%\begin{proof}
%The argument follows from standard computations see.................
%\end{proof}

\begin{remark}
In particular, if $i\neq j,k$, $[N_j,[N_i,N_j]]=[N_k,[N_i,N_k]]=Z_i$ and we also deduce that if $i\neq j,k$, then $[N_i,[N_j,N_k]]=0$.
\end{remark}

\begin{example}
We build a realization of the above nilpotent approximation.
Set $\mathfrak{n}_1=\mathfrak{n}_3=\R^n$, $\mathfrak{n}_2=\so(n)$,
$\mathfrak{n}=\mathfrak{n}_1\oplus\mathfrak{n}_2\oplus\mathfrak{n}_3$.
Write $e_i$, $i=1,\dots,n$ for the canonical basis of $\R^n$
and define the following brackets by
\[
[(e_i,0,0),(e_j,0,0)]&:=(0,e_i\wedge e_j,0) \\
[(e_i,0,0),(0,e_j\wedge e_k,0)]&=-[(0,e_j\wedge e_k,0),(e_i,0,0)]:=(0,0,-(e_j\wedge e_k)e_i)\\
&=(0,0,-\delta_{ik}e_j+\delta_{ij}e_k) \\
[\mathfrak{n}_3,\mathfrak{n}]&:=0
\]
One easily checks that the Jacobi identity holds.
In this example, for $1\leq i\leq n$, one has $N_i=(e_i,0,0)$, $Z_i=(0,0,e_i)$.
\end{example}

%%%%%%%%%%%%%%%%%%%%%%%%%%%
\subsection{Non-Flatness of the Rolling Distribution}
%%%%%%%%%%%%%%%%%%%%%%%%%%%

The main theorem of this section can be stated as follows.

\begin{theorem}\label{th:flat}
If $n\geq 3$ and $K\neq\hat{K}$, then $\RDist$ is not flat.
\end{theorem}

\vspace{0.5cm}

We argue by contradiction and develop the argument in several steps, which are stated as lemmas below.
Hence suppose from now on that the distribution $\RDist$ is flat,
i.e., that it is locally equivalent to its nilpotent approximation.
Thus, given any $q_0\in Q$,
there is a neighborhood $O$ of $q_0$
and $W_i\in\Gamma(\pi_Q|_O,\pi_{TM})$
such that $\LRD(W_i)$, $i=1,\dots,n$, is the basis of the Lie algebra of
the nilpotent approximation of $\RDist$.
We will concentrate our attention on the neighbourhood $O$.

Since $W_i$, $i=1,\dots,n$, are linearly independent on $O$,
\[
\LRD(W_i)W_j-\LRD(W_j)W_i=\sum_{a} \alpha^a_{ij} W_a,
\]
for unique $\alpha^k_{ij}\in\Cinf(O)$
and thus
\[
[\LRD(W_i),\LRD(W_j)]=\sum_a \alpha^a_{ij}\LRD(W_a)+\kappa\nu(A(W_i\wedge W_j)).
\]

\begin{lemma}
The following hold:
\[
g(W_i,W_j)=0,\quad \forall i\neq j \\
\n{W_i}_g=\n{W_j}_g,\quad \forall i,j.
\]
\end{lemma}

\begin{proof}
We need to compute
\[
%[\LRD(W_k),[\LRD(W_i),\LRD(W_j)]]=-\kappa \LNSD(A(W_i\wedge W_j)W_k)+\mathrm{mod}(\LRD,\nu).
& [\LRD(W_k),[\LRD(W_i),\LRD(W_j)]] \\
&=-\kappa \LNSD(A(W_i\wedge W_j)W_k)
+\sum_a (\LRD(W_k)\alpha^a_{ij})\LRD(W_a)
+\sum_{a,b} \alpha^a_{ij} \alpha^b_{ka} \LRD(W_b) \\
&+\kappa\sum_a \alpha^a_{ij} \nu(A(W_k\wedge W_a))
+\kappa\nu(A\LRD(W_k)(W_i\wedge W_j))
-\kappa \LRD(\nu(A(W_i\wedge W_j))W_k).
%+\mathrm{mod}(\LRD,\nu).
\]
On the other hand, by the structure of the nilpotent approximation,
\[
[\LRD(W_i),[\LRD(W_j),\LRD(W_i)]]=[\LRD(W_k),[\LRD(W_j),\LRD(W_k)]],
\]
for all $i,k\neq j$.
This implies that if $i,k\neq j$,
\[
(W_i\wedge W_j)W_i=(W_k\wedge W_j)W_k,
\]
that is
\[
g(W_i,W_j)W_i-g(W_i,W_i)W_j=g(W_k,W_j)W_k-g(W_k,W_k)W_j.
\]
The claim of the lemma follows from this.
\end{proof}

According to the above lemma, the squared norms $\n{W_i}^2_g$ are all the same, for $i=1,\dots,n$, and we use $\beta\in\Cinf(O)$ to denote that common value.
Thus, we have
\[
g(W_i,W_j)=\beta\delta_{ij},\quad \forall i,j.
\]
Notice that $\beta$ never vanishes on $O$.

\begin{lemma}
For all $i,j,k$, we have $\alpha^k_{ij}=0$.
\end{lemma}

\begin{proof}
Since the nilpotent approximation of $\RDist$ has step 3, we have
\[
0&=[[\LRD(W_i),\LRD(W_j)],[\LRD(W_k),\LRD(W_l)]] \\
&=\left[\sum_a \alpha^a_{ij}\LRD(W_a)+\kappa\nu(A(W_i\wedge W_j)),\sum_b \alpha^b_{kl}\LRD(W_b)+\kappa\nu(A(W_k\wedge W_l))\right] \\
%=&\alpha^a_{ij}\alpha^b_{kl}\alpha^t_{ab} \LRD(W_t)+\alpha^a_{ij}(W_a \alpha^b_{kl}) \LRD(W_b)
%-\alpha^b_{kl} (W_b \alpha^a_{ij})\LRD(W_a)+\dots \\
&=-\kappa\LNSD(\sum_a \alpha^a_{ij} A(W_k\wedge W_l)W_a-\sum_b \alpha^b_{kl} A(W_i\wedge W_j)W_b)\qquad\mathrm{mod}(\LRD,\nu),
\]
i.e. for all $i,j,k,l$,
\[
0&=\sum_a \big(\alpha^a_{ij} (W_k\wedge W_l)W_a- \alpha^a_{kl} (W_i\wedge W_j)W_a\big) \\
&=\sum_a \big(\alpha^a_{ij} g(W_l,W_a)W_k-\alpha^a_{ij} g(W_k,W_a)W_l
-\alpha^a_{kl} g(W_j,W_a)W_i+\alpha^a_{kl} g(W_i,W_a)W_j\big) \\
&=\beta (\alpha^l_{ij} W_k-\alpha^k_{ij}W_l-\alpha^j_{kl}W_i+\alpha^i_{kl}W_j).
%=&\alpha^l_{ij} W_k-\alpha^k_{ij} W_l-\alpha^j_{kl} W_i+\alpha^i_{kl} W_j. 
\]
Since $\beta\neq 0$ on $O$, we have arrived at the equation
\[
\alpha^l_{ij} W_k-\alpha^k_{ij}W_l-\alpha^j_{kl}W_i+\alpha^i_{kl}W_j=0,
\]
which holds for all $i,j,k,l$.

Fix any $i,j,k$ which are distinct one from the other. First taking $l=i$, we get
\[
\alpha^i_{ij} W_k-(\alpha^k_{ij}+\alpha^j_{ki})W_i+\alpha^i_{ki}W_j=0,
\]
and hence
\[
\alpha^i_{ij}=0,\quad \alpha^k_{ij}=-\alpha^j_{ki}.
\]
On the other hand, setting $l=j$, we obtain
\[
\alpha^j_{ij} W_k-\alpha^j_{kj}W_i+(\alpha^i_{kj}-\alpha^k_{ij})W_j=0,
\]
from which the only new relation that we get is
\[
\alpha^i_{kj}=\alpha^k_{ij}.
\]
Recalling also that $\alpha^a_{bc}=-\alpha^a_{cb}$, for all$a,b,c$,
we may now compute that
\[
\alpha^k_{ij}=\alpha^i_{kj}=-\alpha^i_{jk}=-\alpha^j_{ik}=\alpha^j_{ki}=-\alpha^k_{ij}
\]
which means that $\alpha^i_{jk}=0$ whenever $i,j,k$ are 2 by 2 distinct.

Now of $i\neq j$, we showed above that $\alpha^i_{ij}=0$
and hence $\alpha^i_{ji}=-\alpha^i_{ij}=0$.
Since also, $\alpha^j_{ii}=0$,
we may conclude that $\alpha^a_{bc}=0$ for all $a,b,c$.
%Assume that $n\geq 4$. Then given $i,j,k$,
%we can choose $l$ distinct from $i,j,k$,
%then we can deduce from the above that on $O$,
%\[
%\sum_a \alpha^a_{ij} g(W_k,W_a)=0,\quad \forall i,j,k
%\]
%i.e.
%\[
%\sum_a \alpha^a_{ij} W_a=0,\quad \forall i,j.
%\]
%This implies the claim by the linear independence of $W_i$s.
\end{proof}

The previous lemma implies that, for all $i,j$,
\[
[\LRD(W_i),\LRD(W_j)]=\kappa\nu(A(W_i\wedge W_j)).
\]

\begin{lemma}
For all $k\neq i,j$,
\[
\LRD(W_k)(W_i\wedge W_j)=0.
\]

%The following relations hold:
%\[
%& \LRD(W_k)(W_i\wedge W_j)=0,\quad \forall k\neq i,j \\
%& \nu(A(W_i\wedge W_j))W_k=0,\quad \forall k\neq i,j \\
%& \LRD(W_i)(W_j\wedge W_i)=\LRD(W_k)(W_j\wedge W_k),\quad \forall i,k\neq j. \\
%& \nu(A(W_j\wedge W_i))W_i=\nu(A(W_j\wedge W_k))W_k,\quad \forall i,k\neq j.
%\]
\end{lemma}

\begin{proof}
For all $i,j,k$,
\[
&[\LRD(W_k),[\LRD(W_i),\LRD(W_j)]] \\
&=-\kappa \LNSD(A(W_i\wedge W_j)W_k)
+\kappa\nu(A\LRD(W_k)(W_i\wedge W_j))
-\kappa \LRD(\nu(A(W_i\wedge W_j))W_k)
\]
On the other hand, by the properties of the nilpotent approximation,
if $k\neq i,j$, 
\[
[\LRD(W_k),[\LRD(W_i),\LRD(W_j)]]=0,
\]
which implies the claim.
%which gives the first two relations in the statement of the lemma.
%On the other hand, for all $i,k\neq j$,
%\[
%[\LRD(W_i),[\LRD(W_j),\LRD(W_i)]]=[\LRD(W_k),[\LRD(W_j),\LRD(W_k)]]
%\]
%i.e.
%\[
%& \LRD(W_i)(W_j\wedge W_i))=\LRD(W_k)(W_j\wedge W_k)) \\
%& \nu(A(W_j\wedge W_i))W_i=\nu(A(W_j\wedge W_k))W_k.
%\]
\end{proof}

%\begin{proof}
%Let $i,j,k,l$ be given  such that $i\neq j$, $k\neq l$ and compute
%\[
%\nu(A(W_i\wedge W_j))(W_k\wedge W_l)
%=(\nu(A(W_i\wedge W_j)) W_k)\wedge W_l+W_k\wedge \nu(A(W_i\wedge W_j)) W_l.
%\]
%If $\{k,l\}\cap \{i,j\}=\emptyset$, then by the second equation of the previous lemma
%\[
%\nu(A(W_i\wedge W_j))(W_k\wedge W_l)=0.
%\]
%On the other hand, if $k\neq i,j$ and $l=i$, then
%by the last equation of the previous lemma
%\[
%\nu(A(W_i\wedge W_j))(W_k\wedge W_l)
%=W_k\wedge \nu(A(W_i\wedge W_j)) W_i
%=W_k\wedge \nu(A(W_k\wedge W_j)) W_k
%\]
%
%
%
%\[
%0=&[\nu((\cdot)(W_i\wedge W_j)),\nu((\cdot)(W_k\wedge W_l)] \\
%=&\nu([W_i\wedge W_j,W_k\wedge W_l]_{\so})
%+\nu\big(A\nu(A(W_i\wedge W_j))(W_k\wedge W_l)-A\nu(A(W_k\wedge W_l))(W_i\wedge W_j)\big)
%\]
%\end{proof}

\begin{lemma}
The function $\beta$ is (locally) constant on $O\subset Q$.
\end{lemma}

\begin{proof}
Given $i,j,k$ all distinct from one another, we have on the first hand,
\[
\LRD(W_k)((W_i\wedge W_j)W_j)=\LRD(W_k)(\beta W_i)=(\LRD(W_k)\beta) W_i+\beta\LRD(W_k)W_i,
\]
and on the other hand, since $\LRD(W_k)(W_i\wedge W_j)=0$ by the previous lemma,
we have
\[
\LRD(W_k)((W_i\wedge W_j)W_j)
&=(W_i\wedge W_j)\LRD(W_k)W_j \\
&=g(\LRD(W_k)W_j,W_j)W_i-g(\LRD(W_k)W_i,W_j)W_j.
\]
But
\[
& g(W_j,W_j)=\beta\quad\Longrightarrow\quad g(\LRD(W_k)W_j,W_j)=\frac{1}{2}\LRD(W_k)\beta, \\
%& g(W_i,W_j)=0\quad\Longrightarrow\quad g(\LRD(W_k)W_j,W_i)=-g(\LRD(W_k)W_i,W_j)
\]
and so
\[
\LRD(W_k)((W_i\wedge W_j)W_j)=\frac{1}{2}(\LRD(W_k)\beta)W_i-g(\LRD(W_k)W_j,W_i)W_j.
\]
Thus we have shown that
\[
\frac{1}{2}(\LRD(W_k)\beta) W_i+\beta\LRD(W_k)W_i=-g(\LRD(W_k)W_j,W_i)W_j,
\]
whenever $i,j,k$ are all distinct.
Taking inner product with respect to $W_i$ we get
\[
0=\frac{1}{2}(\LRD(W_k)\beta) \beta+\beta g(\LRD(W_k)W_i,W_i)=\beta \LRD(W_k)\beta,
\]
which shows that $\beta$ is locally constant on $O$,
because it does not vanish on $O$.
\end{proof}

The result and the proof of the last lemma implies that for all $i,j,k$ distinct,
\begin{align}\label{eq:LWkWi}
\beta\LRD(W_k)W_i=-g(\LRD(W_k)W_j,W_i)W_j.
\end{align}
This observation allows us to prove the next lemma.

\begin{lemma}
For all $i,j$,
\[
\LRD(W_i)W_j=0.
\]
\end{lemma}

\begin{proof}
By \eqref{eq:LWkWi}, for all $i,j,k$ distinct, we have
\[
\LRD(W_k)W_i=a^j_{ki}W_j,
\]
where $a^j_{ki}:=-\beta^{-1}g(\LRD(W_k)W_j,W_i)W_j$.
Because $\LRD(W_k)W_i=\LRD(W_i)W_k$, we deduce that $a^j_{ki}=a^j_{ik}$.
On the other hand,
\[
\beta a^j_{ki}&=a_{ki}g(W_j,W_j)=g(\LRD(W_k)W_i,W_j)\\
&=-g(W_i,\LRD(W_k)W_j)=-a^i_{kj}g(W_i,W_i)=-\beta a^i_{kj},
\]
i.e. $a^j_{ki}=-a^i_{kj}$.
Therefore,
\[
a^j_{ki}=-a^i_{kj}=-a^i_{jk}=a^k_{ji}=a^k_{ij}=-a^j_{ki},
\]
i.e. $a^j_{ki}=0$, for all $i,j,k$ distinct.
This proves that for all $i\neq j$,
\[
\LRD(W_i)W_j=0.
\]

Finally, since $\beta$ is locally constant, $g(\LRD(W_i)W_i,W_i)=0$
and if $j\neq i$,
\[
g(\LRD(W_i)W_i,W_j)=-g(W_i,\LRD(W_i)W_j)=0
\]
and hence
\begin{equation*}
\LRD(W_i)W_i=0,\quad \forall i.\qedhere
\end{equation*}
\end{proof}

The previous fact has a natural useful consequence.

\begin{lemma}
For all $i,j,k$, one has
\[
\nu(A(W_i\wedge W_j))W_k=\frac{\beta K}{\kappa}(\delta_{jk}W_i-\delta_{ik}W_j).
\]
\end{lemma}

\begin{proof}
By the previous lemma, for all $i,j,k$,
\[
0&=\LRD(W_i)\LRD(W_j)W_k-\LRD(W_j)\LRD(W_j)W_k \\
&=\LRD(\LRD(W_i)W_j-\LRD(W_j)W_i)W_k+\kappa \nu(A(W_i\wedge W_j))W_k\\
&+R^{\ol{\nabla}}((W_i,AW_i),(W_j,AW_j))(W_k,0) \\
&=\kappa\nu(A(W_i\wedge W_j))W_k-K(W_i\wedge W_j)W_k.\qedhere
\]

\end{proof}

We are now in position to finish the proof of Theorem \ref{th:flat}.

%\begin{lemma}
%Either $K=0$ or $W_j=0$, $i=1,\dots,n$.
%\end{lemma}
%

\begin{proof}[Proof of Theorem \ref{th:flat}]
By the last lemma, we have for all $i,j,k,l,m$,
\begin{multline*}
0=[\nu((\cdot)(W_l\wedge W_m),\nu((\cdot)(W_i\wedge W_j))]W_k \\
=\frac{\beta K}{\kappa}\big(\nu(A(W_l\wedge W_m)(\delta_{jk}W_i-\delta_{ik}W_j)
-\nu(A(W_i\wedge W_j)(\delta_{mk}W_l-\delta_{lk}W_{m})\big) \\
=\left(\frac{\beta K}{\kappa}\right)^2\big(
\delta_{jk}\delta_{mi}W_l-\delta_{jk}\delta_{li}W_m
-\delta_{ik}\delta_{mj}W_l+\delta_{ik}\delta_{lj}W_m \\
-\delta_{mk}\delta_{jl}W_i+\delta_{mk}\delta_{il}W_j
+\delta_{lk}\delta_{jm}W_i-\delta_{lk}\delta_{im}W_j\big)
\end{multline*}

Suppose now that $i,j,k$ are distinct and take $l=i$, $m=k$.
Then the above reduces to
\[
0=&\left(\frac{\beta K}{\kappa}\right)^2W_j,
\]
which means that either $W_j=0$ for all $j=1,\dots,n$
or $K=0$.
The former is absurd, so we must have $K=0$.
If one repeats the above argument with the roles of $(M,g)$ and $(\hat{M},\hat{g})$
reversed, we will also obtain $\hat{K}=0$,
which contradicts the assumption that $K\neq\hat{K}$.
\end{proof}

\end{document}